\documentclass[11pt]{article}
\usepackage{amssymb}
\usepackage{graphicx}
\usepackage[all]{xy}
\usepackage{color}
\usepackage{latexsym}
\usepackage{amssymb,amsmath,amstext,amsfonts}
\usepackage{epsfig}
\usepackage{graphics}
\usepackage{amsthm}
\usepackage{verbatim}
\usepackage[colorlinks=true, urlcolor=blue, linkcolor=blue, citecolor=blue]{hyperref}
\usepackage[margin=1in]{geometry}
\usepackage{tikz, float}
\usepackage{natbib}
 \usetikzlibrary {positioning}

\def\NZQ{\mathbb}               

\def\QQ{{\NZQ Q}}

\def\RR{{\NZQ R}}

\def\EE{{\mathbb E}}
\def\frk{\mathfrak}               

\def\Phi{{\frk n}}
\def\Phi{{\frk N}}

\def\Oc{{\mathcal O}}

\def\R{{\mathbb R}}

\def\opn#1#2{\def#1{\operatorname{#2}}} 
\opn\chara{char} \opn\length{\ell} \opn\pd{pd} \opn\rk{rk}
\opn\projdim{proj\,dim} \opn\injdim{inj\,dim} \opn\rank{rank}
\opn\depth{depth} \opn\grade{grade} \opn\height{height}
\opn\embdim{emb\,dim} \opn\codim{codim}

\opn\Tr{Tr} \opn\bigrank{big\,rank}
\opn\superheight{superheight}\opn\lcm{lcm}
\opn\trdeg{tr\,deg}
\opn\reg{reg} \opn\lreg{lreg} \opn\ini{in} \opn\lpd{lpd}
\opn\size{size} \opn\sdepth{sdepth}
\opn\link{link}\opn\fdepth{fdepth}\opn\lex{lex}
\opn\LM{LM}
\opn\LC{LC}
\opn\NF{NF}
\opn\Merge{Merge}
\opn\sgn{sgn}
\opn\div{div} \opn\Div{Div} \opn\cl{cl} \opn\Pic{Pic}
\opn\Prin{Prin}
\opn\op{op}
\opn\indeg{indeg} \opn\outdeg{outdeg}
\opn\red{red}

\opn\Spec{Spec} \opn\Supp{Supp} \opn\supp{supp} \opn\Sing{Sing}
\opn\Ass{Ass} \opn\Min{Min}\opn\Mon{Mon}
\opn\Ann{Ann} \opn\Rad{Rad} \opn\Soc{Soc}
\opn\Im{Im}
  \opn\Ker{Ker} \opn\Coker{Coker} \opn\Am{Am}
 \opn\Hom{Hom} \opn\Tor{Tor} \opn\Ext{Ext} \opn\End{End}
 \opn\Aut{Aut} \opn\id{id}

\opn\nat{nat}
\opn\pff{pf}
\opn\Pf{Pf} \opn\GL{GL} \opn\SL{SL} \opn\mod{mod} \opn\ord{ord}
\opn\Gin{Gin} \opn\Hilb{Hilb}\opn\sort{sort}
\opn\span{span}
\opn\Image{Image}
\theoremstyle{plain}%
\newtheorem{theorem}{Theorem}
\numberwithin{theorem}{section}
\newtheorem{proposition}[theorem]{Proposition}
\newtheorem{example}[theorem]{Example}
\newtheorem{lemma}[theorem]{Lemma}
\newtheorem{corollary}[theorem]{Corollary}

\newtheorem{remark}[theorem]{Remark}
\newtheorem{conjecture}[theorem]{Conjecture}

\def\R{\mathbb{R}}

\allowdisplaybreaks

\tikzstyle{Cwhite}=[scale = .8,circle, fill = white, minimum size=3mm] 
\tikzstyle{Cgray}=[scale = .4,circle, fill = gray, minimum size=3mm] 
\tikzstyle{Cblack2}=[scale = .4,circle, fill = black, minimum size=5mm] 
\tikzstyle{Cblack}=[scale = .7,circle, fill = black, minimum size=3mm]
\tikzstyle{C0}=[scale = .9,circle, fill = black!0, inner sep = 0pt, minimum size=3mm]
\tikzstyle{C1}=[scale = .7,circle, fill = black!0, inner sep = 0pt, minimum size=3mm]
\tikzstyle{Cred}=[scale = .4,circle, fill = red, minimum size=3mm] 

\date{}
\begin{document}

\title{\bf A Family of Quasisymmetry Models}

\author{Maria Kateri, Fatemeh Mohammadi and Bernd Sturmfels}

\maketitle

 \begin{abstract} \noindent
We present a one-parameter family of  models for
square contingency tables
that interpolates between 
the classical quasisymmetry model and
its Pearsonian analogue.
Algebraically, this corresponds to deformations of
toric ideals associated with graphs.
Our discussion of the statistical issues centers around
maximum likelihood estimation.
\end{abstract}

\bigskip

\noindent
{\bf Keywords:} square contingency tables; algebraic statistics; toric models; linear models; \\ maximum likelihood estimation; $\phi$-divergence.

\section{Introduction}

Consider a square contingency table with commensurable row and column classification variables $X$ and $Y$. Such tables can
 arise from cross-classifying repeated measurements of a categorical response variable. They are common in panel and social mobility studies. One of the most cited examples,  taken from \cite{S}, 
is shown in Table \ref{vision_data}. It cross-classifies $7477$ 
female subjects
according to  the
distance vision levels of their right and left~eyes.

\begin{table}[h]
\label{vision_data}
\begin{center}
\qquad \qquad \qquad \qquad {\bf Left Eye Grade}  \\ 
\begin{tabular}{c@{~~~}c@{~~~}c@{~~~}c@{~~~}c@{~~~}c}
\hline 
 {\bf Right Eye Grade} & best & second & third & worst & \\
\hline
best & 1520  & 266 & 124 & 66 & \\ 
second & 234 & 1512 &  432 &  78 & \\
third & 117 & 362 & 1772 & 205 & \\
worst & 36 & 82 & 179 & 492 & \\
 \hline
\end{tabular}
\end{center}
\vspace{-0.1in}
\caption{Cross classification of 7477 women by unaided distance 
vision of right and left eyes.  } 
\end{table}

The most parsimonious  model for such tables is the
 symmetry (S) model, due to \cite{B}. 
While the S model is easy to interpret, it is too restrictive and rarely fits well.
 An important model  that 
is often of adequate fit
is the quasi-symmetry (QS) model of \cite{C}. \cite{KP} studied  the QS model from the information-theoretic point of view and generalized it to a family of models based on the $\phi$-divergence \citep{Pardo}. In their framework,
 classical QS is closest to the S model under the Kullback-Leibler divergence.
 However, by changing the divergence used to measure 
proximity of distributions, alternative QS models are 
found. For instance, the
Pearsonian divergence yields the  Pearsonian QS model. 
For the data in Table \ref{vision_data}, 
\cite{BFH} applied the QS model,
while \cite{KP} applied  the Pearsonian QS model,
and here these two lead to estimates of similar fit.
However, there are other data sets where only
one of them performs well. Our goal  is to link these 
two models. We shall construct  a one-parameter family of QS models 
that connects these two. In this way, more options for data analysis
are available. 
In case of a single square contingency table, the optimal choice of this model parameter would be of interest. However, the more interesting practical application lies in analyzing and comparing independent square tables of the same set-up, when they cannot be modeled adequately all by the same (classical or Pearsonian) QS model. For example, consider the same panel study carried out at two independent centers, with one of them being modeled only by the classical QS and the other only by the Pearsonian QS. In this scenario, the two fitted models are not as comparable as we would like. Our approach furnishes in-between compromise models. 

Our family exhibits interesting properties when viewed from
the perspective of algebraic statistics \citep{DSS}.
It interpolates between 
two fundamental classes of discrete variable models,
namely, toric models and linear models \citep[\S 1.2]{ASCB}.
Indeed, the QS model is toric, and its Markov basis
is well-known, by work of  \cite{Rap}
and  Latunszynski-Trenado \citep[\S 6.2]{DSS}.
The Pearsonian QS model reduces to a linear
model, specified by the second factors in (\ref{QS_1}).
Its ML degree is the number of bounded regions in the
arrangement of hyperplanes $\{a_i  - a_j = 1 \}$,
   by Varchenko's formula \citep[Theorem 1.5]{ASCB}.

This paper is organized as follows. Our parametric family 
of QS models is introduced in Section 2. In Section 3 we derive
the implicit representation of our model by polynomial
equations in the cell entries. That section is written in the
algebraic language of ideals and varieties. It will be of independent interest
to scholars in combinatorial commutative algebra \citep{MillerSturmfels, GrobnerBasis}.
Maximum likelihood estimation (MLE) and the fit of the model are 
discussed in Section~4. 
Section~5 examines a natural submodel given by independence constraints.
 Section~6 discusses statistical applications and presents
computations with concrete data sets.
Section~7 offers an information-theoretic characterization in terms of $\phi$-divergence, following \cite{KP} and \cite{Pardo}.

\section{Quasisymmetry Models}

We consider models for square contingency tables of format $I \times I$.
Probability tables ${\bf p} = (p_{ij})$ are points in the simplex  $\Delta_{I^2-1}$.
Here $p_{ij }$ is the probability that an
observation falls in the $(i,  j)$ cell. We write
  ${\bf n} = (n_{ij})$ for the table of observed frequencies.
The model of symmetry (S) is 
\begin{equation}\label{S}
\qquad p_{ij}=s_{ij} \ \ \ \ \text{with parameters} \ \ s_{ij}=s_{ji} \ \ \ \ 
{\rm for} \,\,1 \leq i \leq j \leq I .
\end{equation}
Here, and in what follows, the table $(s_{ij})$ is non-negative and its entries sum to $1$.
Geometrically, the S model is a simplex of dimension
$\binom{I+1}{2}-1$ inside the ambient
probability simplex $\Delta_{I^2-1}$.
The classical {\em QS  model} can be defined, as a model of divergence from S, by
\begin{equation}\label{QS_0}
p_{ij}=s_{ij}\frac{2c_i}{c_i+c_j} \ , \  \ \ \  i,j=1, \ldots, I .
\end{equation}
The {\em Pearsonian QS model} is defined by the parametrization
\begin{equation}\label{QS_1}
p_{ij}=s_{ij}(1+a_i-a_j) \ , \  \ \ \ i,j=1, \ldots, I .
\end{equation}
Both models are semialgebraic subsets of dimension $\binom{I+1}{2}+I-2$ in
the simplex $\Delta_{I^2-1}$. The S model is the subset obtained 
respectively for  $c_1 = \cdots = c_I$ in (\ref{QS_0})
or $a_1 = \cdots = a_I$ in (\ref{QS_1}).

We here study the following quasisymmetry model  (${\rm QS}_t$),
where  $t\in [0, 1]$  is a parameter:
\begin{equation}\label{QS_t_a}
p_{ij}=s_{ij}\left(1+\frac{(1+t)(a_i-a_j)}{2+(1-t)(a_i+a_j)}\right) \ , \ \ \  i\neq j, \ \ \  i,j=1, \ldots, I .
\end{equation}
In all three models, the matrix entries on the diagonal are set to $p_{ii} = s_{ii}$ for $i=1,\ldots,I$.
For $t = 1$, the model (\ref{QS_t_a}) specializes to the 
Pearsonian QS model (\ref{QS_1}). For $t=0$, it
specializes to the QS model (\ref{QS_0}),
if we set $a_i = c_i - 1$. 
The parameters
$a_i$ will be assumed to satisfy the restriction
\begin{eqnarray}\label{eq:a_i parameters}
t \cdot {\rm max}_i a_i-{\rm min}_i a_i \, \leq \,1.
\end{eqnarray}
Since we had assumed $0 \leq s_{ij} \leq 1/2$,
the constraint
\eqref{eq:a_i parameters} on the $a_i$ ensures that
the $p_{ij}$ are probabilities
(i.e.~lie in the interval $[0,1]$).
Furthermore, if we  change the parameters via
\[s_{ii} = x_{ii}\, \,\,\,\hbox{for}\, \,i=j, \quad \hbox{and} \quad
s_{ij}=x_{ij}\left(1+(1-t)\frac{a_i+a_j}{2}\right) \ \ \,  \hbox{for} \,\, i \not= j, \]
then the model (${\rm QS}_t$), defined in (\ref{QS_t_a}), is rewritten in the simpler form
\begin{equation}\label{QS_t}
p_{ij}=x_{ij}(1+a_i-t a_j) \ , \ \ \  i\neq j, \ \ \  i,j=1, \ldots, I .
\end{equation}
Note that $x_{i+} = \sum_{j=1}^I x_{ij} = \sum_{j=1}^I x_{ji} = x_{+i}$, since the table $(x_{ij})$
is also symmetric.
For $t=1$, the probabilities defined by (\ref{QS_t}) satisfy $\sum_{i=1}^I p_{ij}=1$ for all $j$.
In order to ensure that $\sum_{i=1}^I p_{ij}=1$ for $t\neq 1$ as well, 
we use the `weighted sum to zero' constraint
\begin{equation}\label{constr}
\sum_{i=1}^I (x_{i+}-x_{ii})a_i = 0 .
\end{equation}

The expressions (\ref{QS_t_a})  and  (\ref{QS_t}) 
are equivalent. Whether one or the other is preferred is a matter of
convenience. Maximum likelihood estimation is easier with
(\ref{QS_t_a}), since the MLEs of the $s_{ij}$ are rational functions
of the observed frequencies $n_{ij}$. The estimates of the $a_i$
depend algebraically on ${\bf n}$, and they generally have
to be computed by an iterative method. In the formulation
(\ref{QS_t}), none of the parameters have estimates 
that are rational in ${\bf n}$.
We shall see this in Section 4. On the other hand,
for our algebraic analysis of the ${\rm QS}_t$ model, it is 
more convenient to use (\ref{QS_t}).

\begin{example} \label{ex:running} \rm
Fix $I= 3$. For any fixed $t$, the model (\ref{QS_t})
is a hypersurface in the simplex $\Delta_8$ of
all $3 \times 3$ probability tables. This hypersurface is the zero set of the cubic polynomial
\begin{equation}
\label{eq:cubic}
 \begin{matrix}  (1+t+t^2)(p_{12}p_{23}p_{31}-p_{21}p_{32}p_{13}) \,+\, \\
t(p_{12}p_{23}p_{13}+p_{12}p_{32}p_{31}+p_{21}p_{23}p_{31}-p_{12}p_{32}p_{13}-p_{21}p_{23}p_{13}-p_{21}p_{32}p_{31}).
\end{matrix} 
\end{equation}
For $t = 0$, we recover the familiar binomial relation
that encodes the cycle of length three \citep[\S 6.2]{DSS}.
Thus, our family of ${\rm QS}_t$ models 
represents a deformation of that Markov basis:
$$ p_{12}p_{23}p_{31}-p_{21}p_{32}p_{13} + O(t) .$$
The generalization of the relation (\ref{eq:cubic}) to higher values of $I$
will be presented in Section 3. \hfill $\diamondsuit$
\end{example}
 
Another characteristic model for square tables with commensurable classification variables is the model of {\it marginal homogeneity} (MH). This is specified by the equations
\begin{equation}\label{MH}
p_{i+}\,=\,p_{+i} \, \quad \hbox{for} \,\,\,\, i=1, \ldots, I .
\end{equation}
The model of symmetry S implies MH and QS, {\it i.e.~}(\ref{QS_0}) with $c_1 = \cdots = c_I$. 
By \citet[\S 8.2.3]{BFH},
 if the models MH and QS hold simultaneously, then S is implied. In symbols,
 $\text S=\text {MH} \cap \text {QS}$.
 This identity is important in that it
 underlines the role of the parameters $c_i$
 in the QS model. These express the contribution of the classification category $i$ to marginal inhomogeneity. We shall prove next that
 the same identity  holds for our generalized ${\rm QS}_t$ model. 

\begin{proposition}\label{prop1}
For any $t \in [0, 1]$, we have $\text S=\text {MH} \cap QS_t$.
\end{proposition}
\begin{proof}
It is straightforward to verify that S implies MH and ${\rm QS}_t$ with $a_i=0$, for all $i$, which leads to $p_{ij}=x_{ij}=s_{ij}$, for all $i,j$.
On the other hand, under ${\rm QS}_t$ as defined by (\ref{QS_t}), we have
\begin{equation}\label{marg_differ}
p_{i+}-p_{+i}\,\,=\,\,(1+t)\biggl(a_i(x_{i+}-x_{ii})-\sum_{j\neq i} a_j x_{ij}\biggr) \, \qquad
\hbox{for} \,\,\,\, i=1, \ldots, I .
\end{equation}
Combining this with MH as in (\ref{MH}), and setting
 $y_i := x_{ii}-x_{i+}$,
the equation (\ref{marg_differ}) implies 
\begin{equation}\label{MH_condition}
\sum_{j\neq i} a_j x_{ij} + a_i y_i \,\,= \,\,0 \, \qquad
\hbox{for} \,\,\,\, i=1, \ldots, I .
\end{equation}
This can be written in
 the matrix form ${\bf B}{\bf a}={\bf 0}$, where ${\bf a}=(a_1, \ldots, a_I)^T$,
 ${\bf x}=(x_{ij})$, and
$$  \qquad \qquad {\bf B} \,= {\bf x}-\text{diag}({\bf x}{\bf 1}) \, = 
\begin{bmatrix}
 & &  & x_{1I} \\
 & \tilde{\bf B} &  & \vdots \\
 &  &  & x_{I-1,I} \\ 
x_{I1} & x_{I2} & \ldots & y_I \end{bmatrix} .
$$
The matrix $\tilde{\bf B}$ is strictly diagonally dominant, provided $|y_i|=x_{i+}-x_{ii}> \sum_{j\neq i}^{I-1} x_{ij}$.
This is ensured if all $x_{il}$ are positive, as in
Remark~\ref{rem:structural}; otherwise a separate argument is needed.

 By the Levy-Desplanques Theorem, the matrix $\tilde{\bf B}$ is invertible and $\rank(\tilde{\bf B})=I-1$.
Hence $\rank({\bf B})=I-1$, since ${\bf B}{\bf 1}={\bf 0}$.
Therefore, all solutions of ${\bf B}{\bf a}={\bf 0}$ have the form
 ${\bf a}=a{\bf 1}$ for some $a \in \R$.
For $t=1$, equation (\ref{QS_t}) now implies $p_{ij}=x_{ij}=s_{ij}$, for all $i, j$. 
For $t\neq 1$, combining (\ref{constr}) with the positivity of $x_{i+}-x_{ii}$, we get $a=0$. Hence 
symmetry S holds and the proof is complete.
\end{proof}
\begin{remark}\label{rem:structural}\rm
Contingency tables with structural zeros, i.e., cells of zero probability, are rare.
If they exist, they usually have a specific pattern (zero diagonal, triangular table). In our set-up it is realistic to assume that there exists an index $j$ such that $p_{ij}>0$ for all $i=1, \ldots, I$. Thus, without loss of generality, we can assume that $p_{iI}>0$ and therefore $x_{iI}>0$ for all $i=1, \ldots, I$.
\end{remark}

\begin{example} \rm ($I=3$)
Marginal homogeneity defines a 
linear space of codimension $2$, via
\begin{eqnarray*}
  p_{11}+p_{12}+p_{13} &=& p_{11}+p_{21}+p_{31},\\
  p_{21}+p_{22}+p_{23} &=& p_{12}+p_{22}+p_{32},\\
  p_{31}+p_{32}+p_{33} &=& p_{13}+p_{23}+p_{33}.
\end{eqnarray*}
Inside that linear subspace, the cubic  (\ref{eq:cubic})
factors into a hyperplane, which is the
 S model $\{p_{12} = p_{21}, \ p_{13} = p_{31}, \ p_{23} = p_{32}\}$,
 and a quadric, which 
 has no points with positive coordinates. \hfill $\diamondsuit$
\end{example}
In the light of Proposition \ref{prop1}, the parameter $a_i$ of the ${\rm QS}_t$ model can be interpreted as the contribution of each category $i$ to 
the {\em marginal inhomogeneity}. By this we mean the
difference of $a_i$ minus the weighted average of all $a_i$'s.
This is the parenthesized expression in the identity
\begin{equation}\label{marg_differ_a}
 p_{i+}-p_{+i}\,\,=\,\,(1+t)x_{i+}\left(a_i-\sum_{j} \frac{x_{ij}}{x_{i+}} a_j \right) \ , \qquad i,j=1, \ldots, I .
\end{equation}

\section{Implicit Equations}

We now examine the quasisymmetry models
${\rm QS}_t$ through the lens
of algebraic statistics \citep{DSS, ASCB, Rap}.
To achieve more generality and flexibility,
we fix an undirected simple graph $G$ with vertex set $\{1,2,\ldots,I\}$.
Let $\mathcal{I}_{G}$ denote the prime ideal of algebraic relations among the
quantities $p_{ij} = x_{ij} (1+a_i-ta_j)$
in (\ref{QS_t}), where $\{i,j\}$ runs over the edge set $E(G)$ of the graph $G$.
The ideal $\mathcal{I}_G$  lives in the polynomial ring
$\mathbb{K}[\,p_{ij},p_{ji}: \{i,j\} \in E(G) \,]$. Here we take
$\mathbb{K} = \mathbb{Q}[[t]]$
to be the local ring of formal Laurent series in one unknown $t$.

Our main result in this section is the derivation of a generating set for the
ideal $\mathcal{I}_G$.
One motivation for studying this ideal is the constrained formulation
of the MLE problem in Section 4.

The model  in Section 2 corresponds to the 
complete graph on $I$ nodes, denoted $G = K_I$.
In particular, for $I=3$, the ideal $\mathcal{I}_{K_3}$ is the
principal ideal generated by the cubic in (\ref{eq:cubic}).
Here we work with arbitrary graphs $G$, not just $K_I$, so as to allow for 
sparseness in the models.
We disregard the `weighted sum to 0' constraint \eqref{constr},
as this does not affect the homogeneous relations in 
$\mathcal{I}_{G}$.

\smallskip

Let $\EE(G)$ denote the set of oriented edges of $G$.
For each edge $\{i,j\}$ in $E(G)$ there are two edges $ij$ and $ji$ in $\EE(G)$.
 So we have $|\EE(G)|=2|E(G)|$.
An {\em orientation} of $G$ is the choice of a subset $\Oc \subset \EE(G)$ such
 that, for each edge $\{i,j\}$ in $E(G)$, {\it either} $ij$ or $ji$ belongs to $\Oc$. 
 An orientation of $G$ is called {\em acyclic} if it contains no directed cycle. 

Let $C$ denote the undirected $n$-cycle, with $E(C) = \{\{1,2\},\{2,3\}, \ldots,\{n,1\}\}$.
Then $C$ has $2^n$ orientations, shown in
 Figure \ref{fig:spanning trees} for  $n=3$.
 Precisely two of these orientations are  {\em cyclic}.
These two directed cycles are  denoted by $o_C$ and $\bar{o}_C$.
Their edge sets are $\EE(o_C)=\{12,23,\ldots,n1\}$ and $\EE(\bar{o}_C)=\{21,32,\ldots,1n\}$.
Any orientation $\delta_C$ of $C$ defines a monomial of degree $n$ via
 $$ p^{\delta_C} \,\,=\prod_{ij\in\EE(\delta_C)}p_{ij}. $$
 We also define the integer ${\rm c}(\delta_C)=2|\EE(o_C)\cap \EE(\delta_C)|-n$.
Note that ${\rm c}(o_C)=n$ and  ${\rm c}(\bar{o}_C)=-n$.
 
We associate with the $n$-cycle $C$ the following polynomial of degree $n$ with $2^n$ terms:
\begin{equation}
\label{eq:cyclepoly}
P^{C}\,\,=\,\,\sum_{\delta_C}  {\rm coeff}(\delta_C) \cdot p^{\delta_C}.
\end{equation}
The sum is over all orientations $\delta_C$ of $C$, and
the coefficients are the  scalars in $\mathbb{K}$ defined by
\begin{equation*}\label{eq:even}
{\rm coeff}(\delta_{C})\,= \,
\begin{cases}
\frac{{\rm c}(\delta_C)}{|{\rm c}(\delta_C)|} \cdot \big(t^{r-\frac{{|{\rm c}(\delta_C)|}}{2}}+t^{r+2-\frac{{|{\rm c}(\delta_C)|}}{2}}+\cdots+t^{r+\frac{{|{\rm c}(\delta_C)|}}{2}-2} \big)\quad \quad \quad\quad {\rm if}\quad  n=2r,
\smallskip
\\
\frac{{\rm c}(\delta_C)}{|{\rm c}(\delta_C)|} \cdot \big( t^{r-\frac{{|{\rm c}(\delta_C)|-1}}{2}}+t^{r+1-\frac{{|{\rm c}(\delta_C)|-1}}{2}}+\cdots+t^{r+\frac{{|{\rm c}(\delta_C)|-1}}{2}-1}\big) \quad\ {\rm if} \quad n=2r-1.
\end{cases}
\end{equation*}

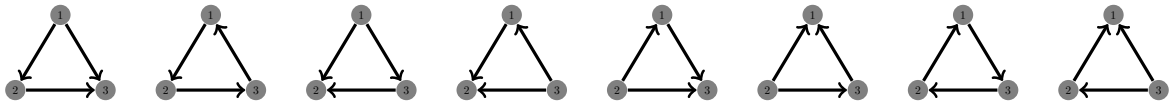
\begin{figure}[h!]
\begin{center}
\begin{tikzpicture} [scale = .20, very thick = 10mm]
\node (n1) at (14,11) [Cgray] {1};
\node (n2) at (11,6)  [Cgray] {2};
\node (n3) at (17,6)  [Cgray] {3};
\foreach \from/\to in {n1/n3,n1/n2,n2/n3}
\draw[black][->] (\from) -- (\to);
  
\node (n1) at (24,11) [Cgray] {1};
\node (n2) at (21,6)  [Cgray] {2};
\node (n3) at (27,6)  [Cgray] {3};
\foreach \from/\to in {n3/n1,n1/n2,n2/n3}
\draw[black][->] (\from) -- (\to);

\node (n1) at (34,11) [Cgray] {1};
\node (n2) at (31,6)  [Cgray] {2};
\node (n3) at (37,6)  [Cgray] {3};
\foreach \from/\to in {n1/n3,n1/n2,n3/n2}
\draw[black][->] (\from) -- (\to);

  \node (n1) at (44,11) [Cgray] {1};
\node (n2) at (41,6)  [Cgray] {2};
\node (n3) at (47,6)  [Cgray] {3};
\foreach \from/\to in {n3/n1,n1/n2,n3/n2}
\draw[black][->] (\from) -- (\to);
------------
\node (n1) at (54,11) [Cgray] {1};
\node (n2) at (51,6)  [Cgray] {2};
\node (n3) at (57,6)  [Cgray] {3};
\foreach \from/\to in {n1/n3,n2/n1,n2/n3}
\draw[black][->] (\from) -- (\to);
 
\node (n1) at (64,11) [Cgray] {1};
\node (n2) at (61,6)  [Cgray] {2};
\node (n3) at (67,6)  [Cgray] {3};
\foreach \from/\to in {n3/n1,n2/n1,n2/n3}
\draw[black][->] (\from) -- (\to);

\node (n1) at (74,11) [Cgray] {1};
\node (n2) at (71,6)  [Cgray] {2};
\node (n3) at (77,6)  [Cgray] {3};
\foreach \from/\to in {n1/n3,n2/n1,n3/n2}
\draw[black][->] (\from) -- (\to);
  
\node (n1) at (84,11) [Cgray] {1};
\node (n2) at (81,6)  [Cgray] {2};
\node (n3) at (87,6)  [Cgray] {3};
\foreach \from/\to in {n3/n1,n2/n1,n3/n2}
\draw[black][->] (\from) -- (\to);
\end{tikzpicture}

\caption{The eight orientations $\delta_1,\delta_2,\ldots,\delta_8$ of $C=K_3$.}
\label{fig:spanning trees}
\vspace{-0.06in}
\end{center}
\end{figure}

\begin{example} \rm
We  consider the cycle $C = K_3$ of length $n=3$.
It has eight orientations, depicted in Figure \ref{fig:spanning trees}.
The corresponding monomials and their coefficients are as follows:

 \begin{center}\begin{tabular}{lll}
$p^{\delta_1}=p_{12}p_{23}p_{13}$ &  \hspace{3ex} ${\rm c}(\delta_1)=1$ &    \hspace{3ex} ${\rm coeff}(\delta_{1})=t$            \\
$p^{\delta_2}=p_{12}p_{23}p_{31}$ &  \hspace{3ex} ${\rm c}(\delta_2)=3$ & \hspace{3ex} ${\rm coeff}(\delta_{2})=1+t+t^2$ \\
$p^{\delta_3}=p_{12}p_{32}p_{13}$ &  \hspace{3ex} ${\rm c}(\delta_3)=-1$ &  \hspace{3ex} $ {\rm coeff}(\delta_{3})=-t$ \\
$p^{\delta_4}=p_{12}p_{32}p_{31}$ &  \hspace{3ex} ${\rm c}(\delta_4)=1$ &    \hspace{3ex} ${\rm coeff}(\delta_{4})=t$ \\
$p^{\delta_5}=p_{21}p_{23}p_{13}$ &  \hspace{3ex} ${\rm c}(\delta_5)=-1$ &  \hspace{3ex} ${\rm coeff}(\delta_{5})=-t$ \\
$p^{\delta_6}=p_{21}p_{23}p_{31}$ &  \hspace{3ex} ${\rm c}(\delta_6)=1$ &   \hspace{3ex} ${\rm coeff}(\delta_{6})=t $\\
$p^{\delta_7}=p_{21}p_{32}p_{13}$ &  \hspace{3ex} ${\rm c}(\delta_7)=-3$ &  \hspace{3ex} $ {\rm coeff}(\delta_{7})=-1-t-t^2$ \\
$p^{\delta_8}=p_{21}p_{32}p_{31}$ &  \hspace{3ex} ${\rm c}(\delta_8)=-1$ &  \hspace{3ex} ${\rm coeff}(\delta_{8})=-t $
\end{tabular}\end{center}
Thus,  the polynomial $P^{C}$ defined in (\ref{eq:cyclepoly}) is the cubic (\ref{eq:cubic}) seen
in Example \ref{ex:running}. \hfill $\diamondsuit$
\end{example}

We define the {\em classical QS model} on the graph $G$ 
by the parametrization (\ref{QS_0})
where $\{i,j\}$ runs over the set $E(G)$ of edges of $G$.
We write $\mathcal{T}_G$ for the ideal of this model.
This is a toric ideal whose Markov basis is
 obtained from the cycle polynomials
$P^C$ by setting $t = 0$:

\begin{lemma} \label{lem:toric}
The ideal $\,\mathcal{T}_G$ has a universal Gr\"obner basis
consisting of the binomials
\begin{equation}
\label{eq:markovbinomial}
\qquad \qquad P^C|_{t=0} \,\, = \,\, p^{o(C)}-p^{\bar{o}(C)} \qquad
\hbox{for all cycles $C$ in $G$.} 
\end{equation}
\end{lemma}

\begin{proof}
The identity in (\ref{eq:markovbinomial})
is straightforward from the definition of $c(\delta_C)$ and ${\rm coeff}(\delta_C)$.
It was shown in \citet[\S 6.2]{DSS} that
the binomials  $p^{o(C)}-p^{\bar{o}(C)}$ form a Markov basis for $QS$.
Since the underlying model matrix is totally unimodular,
the Markov basis is also a Graver basis, and hence it is
a universal Gr\"obner basis, by
\citet[Propositions 4.11 and 8.11]{GrobnerBasis}.
\end{proof}

\begin{example} \rm
For $I=4$, the model ${\rm QS}_t$ corresponds to the
complete graph $K_4$. This graph has seven 
undirected cycles $C$, four of length $3$ and three of length $4$.
Its defining prime ideal $\mathcal{I}_{K_4}$ is generated by
four cubics and three quartics, all of the form $P^C$. For $t = 0$, we
recover the binomials corresponding to the
seven moves that are listed in \citet[\S 5.4, page 395]{Rap}.
\hfill $\diamondsuit$
\end{example}

This example is explained by the following theorem, which is
 our main result in Section~3.

\begin{theorem} \label{thm:ideal}
The prime ideal $\,\mathcal{I}_G$ of the quasisymmetry model associated with an
undirected graph $G$ is generated by the
cycle polynomials $P^C$ where $C$ runs over all cycles in $G$.
\end{theorem}

\begin{proof}
We begin by proving that $P^C$ lies in $\mathcal{I}_G$.
The image of $P^C$ under the substitution
 $p_{ij} \mapsto x_{ij} (1+a_i-ta_j)$ can be written
 as $\,Q^C\times \prod_{\{i,j\}\in E(C)} x_{ij}  $, where
 $Q^C$ is a polynomial in $\mathbb{K}[a_1,\ldots,a_n]$.
 Since each term  $p^{\delta_C}$ of $P^C$ is divisible
 by either $p_{1n}$ or $p_{n1}$, we can write
 \begin{equation}
 \label{eq:deco1}  Q^C\,\,=\,\,(1+a_1-ta_n)T_{1n}+(1+a_n-ta_1)T_{n1}.
 \end{equation}
We need to show that $Q^C$ is zero. To do this, we shall
establish the following identities:
$$ \begin{matrix}
& T_{1n}  & = & (-1)^{[\frac{n-1}{2}]+1}(t + 1)^{2r-2}(1+a_n-ta_1)\prod_{i=2}^{n-1}(1+a_i-t a_i) \\
\hbox{and} \quad & T_{n1} & = & (-1)^{[\frac{n-1}{2}]}(t +1)^{2r-2}(1+a_1-ta_n)\prod_{i=2}^{n-1}(1+a_i-t a_i).
\end{matrix} $$
To prove these, we shall use the decompositions
$$ \begin{matrix}
& T_{1n} & = & 
(1+a_1-ta_2)T_{1n,12}  + (1+a_2-ta_1)T_{1n,21}  \\
\hbox{and} \quad & 
T_{n1} & = & 
(1+a_1-ta_2)T_{n1,12}  + (1+a_2-ta_1)T_{n1,21}.  \\
\end{matrix} $$
With this notation, we claim that the following holds for a suitable integer $r$:
\begin{itemize}
\item[(i)]
$T_{1n,12}\,=\,(-1)^{[\frac{n-2}{2}]}t(t +1)^{2r-3}(a_2-a_n)\prod_{i=3}^{n-1}(1+a_i-t a_i)$,
\item[(ii)]
$T_{1n,21} \,=\,(-1)^{[\frac{n-2}{2}]}(t +1)^{2r-3}(t^2a_2-t-a_n-1)\prod_{i=3}^{n-1}(1+a_i-t a_i)$.
\end{itemize}
Let $C'$ be the cycle $2-3-\cdots-n-2$. 
In analogy to (\ref{eq:deco1}), we  write
$$Q^{C'}=(1+a_2-ta_n)S_{2n}+(1+a_n-ta_2)S_{n2}.$$ 
Note that for any orientation $\delta_C$ of $C$ in which $1n$ and $12$ belong to $\EE(\delta_C)$,  we have 
\begin{equation*}
c(\delta_C)\,\, =\,\,
\begin{cases}
c(\delta_{C'})-1 \quad\quad\quad\ {\rm if} \quad n2\in\EE(\delta_{C'}),
\smallskip \\
c(\delta_{C'})+1 \quad\quad\quad\ {\rm if}\quad  2n\in\EE(\delta_{C'}).&\\
\end{cases}
\end{equation*}
Also note that $\frac{{\rm c}(\delta_C)}{|{\rm c}(\delta_C)|}=\frac{{\rm c}(\delta_{C'})}{|{\rm c}(\delta_{C'})|}$.
In order to prove (i) we  consider the following two cases:
\medskip

\noindent
{\bf Case 1.} $n = 2r-1$ is an odd number: 
We claim that $T_{1n,12}=t(S_{n2}+S_{2n})$. 
Note that $C'$ is an even cycle with $n-1=2(r-1)$.
The coefficient for $\delta_C$ can be written as 
\begin{eqnarray*}
t\times \frac{{\rm c}(\delta_C)}{|{\rm c}(\delta_C)|}\big((t^{r-1-\frac{{|{\rm c}(\delta_C)|-1}}{2}}+t^{r+1-\frac{{|{\rm c}(\delta_C)|-1}}{2}}+\cdots+t^{r+\frac{{|{\rm c}(\delta_C)|-1}}{2}-2})
&+&\\ (t^{r-\frac{{|{\rm c}(\delta_C)|-1}}{2}}+t^{r+2-\frac{{|{\rm c}(\delta_C)|-1}}{2}}+\cdots+t^{r+\frac{{|{\rm c}(\delta_C)|-1}}{2}-3})\big) .
\end{eqnarray*}
The first summand corresponds to the orientation $\delta_{C'}$ with $n2\in \EE(\delta_{C'})$.
The second summand corresponds to  the orientation $\delta_{C'}$ with  $2n\in \EE(\delta_{C'})$.
By induction on $n$, we have 
$$ 
\begin{matrix}
& S_{2n} & = & (-1)^{[\frac{n-2}{2}]+1}(t + 1)^{2r-4}(1+a_n-ta_2)\prod_{i=3}^{n-1}(1+a_i-t a_i) ,\\
\hbox{and} \quad & 
S_{n2} & = & (-1)^{[\frac{n-2}{2}]}(t +1)^{2r-4}(1+a_2-ta_n)\prod_{i=3}^{n-1}(1+a_i-t a_i) .
\end{matrix}
$$
Since  $-(1+a_n-ta_2)+(1+a_2-ta_n)=(1+t)(a_2-a_n)$, the claim (i) holds for $n$ odd.

\medskip

\noindent
{\bf Case 2.} $n = 2r$ is an even number: We 
will first show that $T_{1n,12}=t(S_{n2}+S_{2n})/(1+t)^2$. 
Here $C'$ is an odd cycle on $n-1=2r-1$ vertices.
The coefficient for $\delta_C$ equals
\begin{eqnarray*}
\frac{t}{(1+t)^2}\times\frac{{\rm c}(\delta_C)}{|{\rm c}(\delta_C)|}
\big(t^{r-\frac{|{\rm c}(\delta_C)|}{2}-1}+2t^{r-\frac{|{\rm c}(\delta_C)|}{2}}+\cdots+2t^{r+\frac{|{\rm c}(\delta_C)|}{2}-2}+t^{r+\frac{|{\rm c}(\delta_C)|}{2}-1}\big).
\end{eqnarray*}
This sum can be decomposed as
\[
(t^{r-\frac{|{\rm c}(\delta_C)|}{2}-1}+t^{r-\frac{|{\rm c}(\delta_C)|}{2}}+\cdots+t^{r+\frac{|{\rm c}(\delta_C)|}{2}-1})\\+
(t^{r-\frac{|{\rm c}(\delta_C)|}{2}}+t^{r-\frac{|{\rm c}(\delta_C)|}{2}+1}+\cdots+t^{r+\frac{|{\rm c}(\delta_C)|}{2}-2}),
\]
where the first summand corresponds to the orientation $\delta_{C'}$ with $n2\in\EE(\delta_{C'})$, and the second summand corresponds to  the orientation $\delta_{C'}$ with $2n\in\EE(\delta_{C'})$. Therefore $T_{1n,12}=\frac{t(S_{n2}+S_{2n})}{(1+t)^2}$.
By induction on $n$, we have
$$
\begin{matrix}
& S_{2n} & = & (-1)^{[\frac{n-2}{2}]+1}(t + 1)^{2r-2}(1+a_n-ta_2)\prod_{i=3}^{n-1}(1+a_i-t a_i) \\
\hbox{and} \quad &
S_{n2} & = & (-1)^{[\frac{n-2}{2}]}(t +1)^{2r-2}(1+a_2-ta_n)\prod_{i=3}^{n-1}(1+a_i-t a_i)
\end{matrix}
$$
Since  $-(1+a_n-ta_2)+(1+a_2-ta_n)=(1+t)(a_2-a_n)$, the result holds for even $n$ as well.

\smallskip

By a similar argument one can  prove (ii).
Now applying (i) and (ii) and the equality 
$$
-(1+a_2-ta_2)(1+a_n-ta_1)(1+t)=(1+a_1-ta_2)(a_2-a_n)t+(1+a_2-ta_1)(t^2a_2-t-a_n-1),
$$
we obtain
\[
T_{1n}\,\,= \,\,(-1)^{[\frac{n-2}{2}]+1}(t +1)^{2r-2}(1+a_n-ta_1)\prod_{i=2}^{n-1}(1+a_i-t a_i)\ .
\]
The identity for $T_{n1}$ is analogous. It follows that $P^C \in \mathcal{I}_G$
for all cycles of $G$.

It remains to be shown that the $P^C$  generate the homogeneous ideal $\mathcal{I}_G$.
Recall that, by Lemma \ref{lem:toric}, the
images of the $P^C$ generate this ideal after we tensor, 
over the local ring $\mathbb{K}$, with the residue field $\mathbb{Q} = \mathbb{K}/\langle t \rangle $.
 Hence, by Nakayama's Lemma, the $P^C$ generate $\mathcal{I}_G$.
\end{proof}

\begin{remark} \rm
In  Theorem \ref{thm:ideal}
we can replace the local ring $\mathbb{K}= \QQ[[t]]$
with the polynomial ring $\QQ[t]$
because no $t$ appears in the leading forms
$(P^C)|_{t=0}$. This ensures that 
$\QQ[t][p_{ij}]$ modulo the ideal $\langle P^C: C\, \hbox{cycle in}\, G \rangle$
is  torsion-free,  hence free, and therefore flat over $\QQ[t]$.
\end{remark}

In statistical applications, the quantity $t$ will always
take on a particular real value. In the remainder of this paper,
we assume $t \in \RR$, and we  identify
$\mathcal{I}_G$ with its image in $\RR[p_{ij}]$.

\begin{corollary}
For any $t \in \RR$, the cycle polynomials $P^C$
generate  the ideal $\,\mathcal{I}_G$ in $\RR[p_{ij}]$.
\end{corollary}

Theorem \ref{thm:ideal} furnishes a  (flat)
degeneration from $\mathcal{I}_G$ to
the toric ideal $\mathcal{T}_G$. 
Geometrically, we view this as a
degeneration of varieties (or semialgebraic sets) from
$t > 0$ to $t = 0$.  Lemma \ref{lem:toric} concerns 
 further degenerations from the toric ideal $\mathcal{T}_G$ to its
 initial monomial ideals
 $\mathcal{M}_G$. Any such $\mathcal{M}_G$ is
 squarefree and serves as a combinatorial model for
both $\mathcal{T}_G$ and $\mathcal{I}_G$.

We describe one particular choice and draw some
combinatorial conclusions.
Fix a term order on $\RR[p_{ij}]$ with the property that
$p_{ij} \succ p_{k\ell}$ whenever $i < k$, or $i=k$ and $j<\ell$.
For any cycle $C$, we label  the two directed orientations $o_C$ and $\bar{o}_C$ so that
$p^{o(C)} \succ p^{\bar{o}(C)} $.
 Fix a spanning tree $T$ of $G$. Let
  $\mathfrak{P}_T$ denote the monomial prime ideal generated by all unknowns $p_{ij}$ where 
$\{i,j\}\in E(G)\backslash E(T)$ and  $p_{ij}$ divides $p^{o_C}$, where $C$ is the unique cycle
in $E(T) \cup \{ \{i,j\}\}$.~The squarefree monomial ideal
\begin{equation}
\label{eq:initideal}
 \mathcal{M}_G \,\,  =\,\,
{\rm in}_\succ(\mathcal{T}_G) \,\, = \,\,
\bigl\langle \,p^{o_C} \,: \,C \,\,\hbox{cycle in} \,\, G \,\bigr\rangle \,\, = \,\,
\bigcap_T \mathfrak{P}_T\ ,
\end{equation}
is obtained by taking the
intersection  over all spanning trees $T$ of $G$.
The simplicial complex with {\em Stanley-Reisner ideal} $\mathcal{M}_G$ is
a regular triangulation of the {\em Lawrence polytope}
of the graph $G$. This triangulation is shellable
and hence our ideals are Cohen-Macaulay.
We record the following fact.

\begin{proposition}
The ideals $\,\mathcal{M}_G, \,\mathcal{T}_G\,$ and $\,\mathcal{I}_G\,$
define varieties of dimension $|E(G)|+I-1$ in affine space, and their 
common degree is the number of spanning trees of the graph $G$.
\end{proposition}

\begin{proof}
Each of the components $\mathfrak{P}_T$ in (\ref{eq:initideal}) has
codimension $|E(G)\backslash E(T)| = |E(G)| {-} I {+} 1$.
\end{proof}

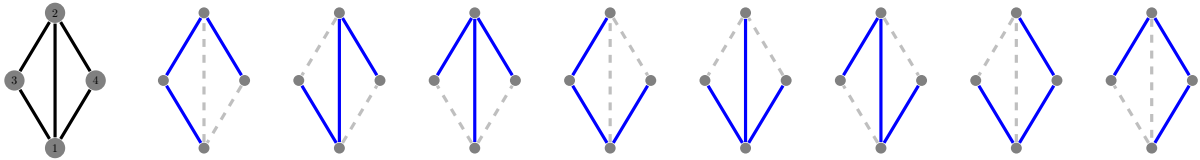
\begin{figure}[h]
\begin{center}
\begin{tikzpicture} [scale = .18, very thick = 10mm]

  \node (n4) at (3,1)  [Cgray] {1};
  \node (n1) at (3,11) [Cgray] {2};
  \node (n2) at (0,6)  [Cgray] {3};
  \node (n3) at (6,6)  [Cgray] {4};
  \foreach \from/\to in {n4/n2,n1/n3}
    \draw[] (\from) -- (\to);
\foreach \from/\to in {n2/n1,n4/n3,n1/n4}
    \draw[] (\from) -- (\to);
    
  \node (n4) at (14,1)  [Cgray] {};
  \node (n1) at (14,11) [Cgray] {};
  \node (n2) at (11,6)  [Cgray] {};
  \node (n3) at (17,6)  [Cgray] {};

\foreach \from/\to in {n2/n1,n1/n3,n4/n2}
    \draw[blue][] (\from) -- (\to);s
 
\foreach \from/\to in {n4/n1,n3/n4}
    \draw[lightgray][dashed] (\from) -- (\to);

    \node (n4) at (24,1)  [Cgray] {};
  \node (n1) at (24,11) [Cgray] {};
  \node (n2) at (21,6)  [Cgray] {};
  \node (n3) at (27,6)  [Cgray] {};
  \foreach \from/\to in {n4/n2,n1/n3,n4/n1}
   \draw[blue][] (\from) -- (\to);

     \foreach \from/\to in {n1/n2,n4/n3}
    \draw[lightgray][dashed] (\from) -- (\to);

    \node (n4) at (34,1)  [Cgray] {};
  \node (n1) at (34,11) [Cgray] {};
  \node (n2) at (31,6)  [Cgray] {};
  \node (n3) at (37,6)  [Cgray] {};
  \foreach \from/\to in {n4/n2,n4/n3}
    \draw[lightgray][dashed] (\from) -- (\to);

    \foreach \from/\to in {n4/n1,n1/n3,n1/n2}
    \draw[blue][] (\from) -- (\to);

  \node (n4) at (44,1)  [Cgray] {};
  \node (n1) at (44,11) [Cgray] {};
  \node (n2) at (41,6)  [Cgray] {};
  \node (n3) at (47,6)  [Cgray] {};
   \foreach \from/\to in {n4/n1,n1/n3}
    \draw[lightgray][dashed] (\from) -- (\to);

    \foreach \from/\to in {n2/n1,n4/n2,n4/n3}
    \draw[blue][] (\from) -- (\to);

    \node (n4) at (54,1)  [Cgray] {};
  \node (n1) at (54,11) [Cgray] {};
  \node (n2) at (51,6)  [Cgray] {};
  \node (n3) at (57,6)  [Cgray] {};
 \foreach \from/\to in {n1/n2,n3/n1}
    \draw[lightgray][dashed] (\from) -- (\to);

    \foreach \from/\to in {n4/n2,n4/n3,n4/n1}
    \draw[blue][] (\from) -- (\to);

    \node (n4) at (64,1)  [Cgray] {};
  \node (n1) at (64,11) [Cgray] {};
  \node (n2) at (61,6)  [Cgray] {};
  \node (n3) at (67,6)  [Cgray] {};
 \foreach \from/\to in {n1/n3,n4/n2}
    \draw[lightgray][dashed] (\from) -- (\to);

    \foreach \from/\to in {n4/n3,n4/n1,n1/n2}
    \draw[blue][] (\from) -- (\to);


   \node (n4) at (74,1)  [Cgray] {};
  \node (n1) at (74,11) [Cgray] {};
  \node (n2) at (71,6)  [Cgray] {};
  \node (n3) at (77,6)  [Cgray] {};
 \foreach \from/\to in {n1/n2,n4/n1}
    \draw[lightgray][dashed] (\from) -- (\to);

    \foreach \from/\to in {n4/n3,n4/n2,n3/n1}
    \draw[blue][] (\from) -- (\to);


 \node (n4) at (84,1)  [Cgray] {};
  \node (n1) at (84,11) [Cgray] {};
  \node (n2) at (81,6)  [Cgray] {};
  \node (n3) at (87,6)  [Cgray] {};
 \foreach \from/\to in {n4/n2,n1/n4}
    \draw[lightgray][dashed] (\from) -- (\to);

    \foreach \from/\to in {n4/n3,n3/n1,n1/n2}
    \draw[blue][] (\from) -- (\to);

\end{tikzpicture}

\caption{A graph $G$ on $I=4$ nodes and its eight spanning trees $T$}
\label{fig:graph48}
\vspace{-0.08in}
\end{center}
\end{figure}

\begin{example}\label{exam:1} \rm
Consider the graph $G$ depicted in Figure~\ref{fig:graph48}.
The associated toric ideal equals
\[
\mathcal{T}_G \,\,\,= \,\,\, \langle  \,
\underline{p_{12}p_{23}p_{31}}-p_{21}p_{32}p_{13} \,,\  
\underline{p_{12}p_{24}p_{41}}-p_{21}p_{42}p_{14} \,,\  
\underline{p_{13}p_{32} p_{24}p_{41}}-p_{31}p_{23} p_{42}p_{14}\,
\rangle.
\]
This has codimension $2$ and degree $8$.
Its (underlined) initial monomial ideal $\mathcal{M}_G$ equals
$$
\langle p_{12},p_{13}  \rangle \,\cap\,
\langle p_{12},p_{32}  \rangle \,\cap\,
\langle p_{12},p_{24}  \rangle \,\cap\,
\langle p_{12},p_{41}  \rangle \,\cap\,
\langle p_{23},p_{41}  \rangle \,\cap\,
\langle p_{23},p_{24}  \rangle \,\cap\,
\langle p_{24},p_{31}  \rangle \,\cap\,
\langle p_{31},p_{41}  \rangle.
$$
These eight monomial prime ideals correspond to the eight spanning trees 
 in Figure~\ref{fig:graph48}.
 The ideal $\mathcal{I}_G$ has three generators, two cubics
 with $8$ terms and one quartic with $16$ terms, as in
 (\ref{eq:cyclepoly}). These are obtained from the Markov basis
 of $\mathcal{T}_G$ by adding additional terms that are divisible by $t$.
  \hfill $\diamondsuit$
\end{example}

\section {Maximum Likelihood Estimation}

 A data table ${\bf n} = (n_{ij})$
 of format $I \times I$  can arise either by multinomial sampling or 
 by sampling from $I^2$ independent Poisson
 distributions, one for each of its cells. 
 In both cases, the log-likelihood function, up to an additive constant, is equal to
 \begin{equation}
\label{eq:loglike}
 \ell_{\bf n}({\bf p}) \quad = \quad \sum_{i=1}^I \sum_{j=1}^I n_{ij} \cdot {\rm log}(p_{ij}). 
 \end{equation}
 Maximum likelihood estimation (MLE) is the problem of
maximizing $\ell_{\bf n}$ over all probability tables ${\bf p} = (p_{ij})$
in the model of interest. For us, that model is 
the quasisymmetry model $(QS_t)$,
where $t$ is a fixed constant in
the  interval $[0,1]$.
This optimization problem can be expressed
in either  constrained form or in  unconstrained form.
The  {\em constrained MLE problem} is written~as
\begin{equation}
\label{eq:constrainedMLE}
 \hbox{Maximize} \,\, \,\ell_{\bf n}({\bf p}) \quad \hbox{subject to} \quad
{\bf p} \,\in \, V(\mathcal{I}_{G} ) \cap \Delta_{I^2-1}, 
\end{equation}
where $G = K_I$ is the complete graph on $I$ nodes,
and $V(\mathcal{I}_G)$ is the zero set of the
cycle polynomials $P^C$ constructed in Section 3.
The {\em unconstrained MLE problem} is written as
\begin{equation}
\label{eq:unconstrainedMLE}
 \hbox{Maximize} \,\, \,\ell_{\bf n}({\bf a} , {\bf s}).
  \end{equation}
 The decision variables in (\ref{eq:unconstrainedMLE}) are the vector ${\bf a} = (a_1,\ldots,a_I)$ and the
symmetric probability matrix ${\bf s} = (s_{ij})$. The
    objective function in (\ref{eq:unconstrainedMLE}) is obtained by substituting
(\ref{QS_t_a}) into (\ref{eq:loglike}). We shall discuss both formulations,
starting with a simple numerical example for the  formulation (\ref{eq:constrainedMLE}).

\begin{example}  
\label{ex:primenumbers} \rm
Let $I = 3, \,t = 2/3$ and consider the data table
$$  \qquad \qquad {\bf n} \,= \begin{bmatrix}
2 & 3 & 5 \\
11 & 13 & 17 \\
19 & 23 & 29 \end{bmatrix} \qquad
\hbox{ with sample size $\,n_{++} = 122$}.
$$
Our aim is to maximize $\ell_{\bf n}({\bf p})$ subject to
the cubic equation (\ref{eq:cubic}) and
$p_{11} + p_{12} + \cdots + p_{33} = 1$.
Using Lagrange multipliers for these two constraints, we derive
the {\em likelihood equations} by way of \citet[Algorithm 2.29]{DSS}.
These polynomial equations in the nine unknowns  $p_{ij}$
have $15$ complex solutions. Two of the complex solutions
are non-real. Of the $13$ real solutions, 
$12$ have at least one negative coordinate.
Only one solution lies 
 in the probability simplex  $\Delta_8$:
\begin{equation}
\label{eq:pvalues}
\begin{matrix}
\hat p_{11} &=& 1/61, &  \qquad
\hat p_{12} &=& 0.0286294, & \qquad
\hat   p_{13} &=& 0.0376289, \\
\hat p_{21} &=& 0.0861247,  &  \qquad
\hat p_{22} &=& 13/122,   & \qquad
\hat   p_{23} &=& 0.1446119, \\
\hat   p_{31} &=& 0.1590924, &  \qquad
\hat   p_{32} &=& 0.1832569, & \qquad
\hat p_{33} &=& 29/122.
\end{matrix}      
\end{equation}
This is the global maximum of the constrained MLE problem
for this instance.
\hfill $\diamondsuit$
\end{example}

The benefit of the constrained formulation is that
we can take advantage of the combinatorial results in Section 3, and
we do not have to deal with issues of identifiability
and singularities arising from the map
(\ref{QS_t_a}).  On the other hand, most statisticians
would prefer  the unconstrained formulation because
this corresponds more directly to the fitting of model parameters to data.

To solve the unconstrained MLE problem (\ref{eq:unconstrainedMLE}),
we take the partial derivations of the objective function 
$\ell_n({\bf a},{\bf s})$ with respect to all model parameters
$a_i$ and $s_{ij}$. The resulting system of equations decouples
into a system for ${\bf a}$ and a system for ${\bf s}$.
The latter is trivial to solve.
Using the requirement that
 the entries of ${\bf s}$ sum to $1$,
it has the closed form solution
\begin{equation}
\label{eq:gets}
 \hat s_{ij} \,\,= \,\,\frac{n_{ij}+n_{ji}}{2n_{++}}  , \quad i,j=1, \ldots, I . 
 \end{equation}
After dividing by $1+t$,
the partial derivatives
of $\ell_{\bf n}({\bf a},{\bf s})$ with respect to $a_1,a_2,\ldots,a_I$  are
\begin{equation}\label{eq:scoreeqn}
\sum_{j =1 \atop j \not= i }^I \frac{
(1+a_j-ta_j)[n_{ij}(1+a_j-ta_i)-n_{ji}(1+a_i-ta_j)]
}{
(1+a_i-t a_j)(1+a_j-ta_i)[2+(1-t)(a_i+a_j)]}
\qquad \hbox{for} \,\, i = 1,2,\ldots,I.
\end{equation}
This system of equations has infinitely many solutions,
because the model ${\rm QS}_t$ is not identifiable.
The general fiber of the map (\ref{QS_t_a}) is a line in
${\bf a}$-space.
Hence only $I-1$ of the $I$ parameters $a_i$
can be estimated. One way to fix this is to simply add the constraint
$\hat a_I = 0$.

\begin{example} 
\label{ex:primenumbers2} \rm
Let us return to the numerical instance in
 Example \ref{ex:primenumbers}. Here we have
\begin{equation}
\label{eq:svalues}
\hat s_{11} = 1/61, \,\,
\hat s_{12} = 7/122, \,\,
\hat s_{13} = 6/61, \,\,
\hat s_{22} = 13/122, \,\, 
\hat s_{23} = 10/61,  \,\,
\hat s_{33} = 29/122.
\end{equation}
The equations (\ref{eq:scoreeqn}) can be solved
in a computer algebra system by clearing denominators and then
saturating the ideal of numerators with respect to those denominators.
As before, there are precisely $15$ complex solutions, of which $13$ are real.
The MLE is given by
\begin{equation}
\label{eq:avalues}
 \hat a_1 = -0.65948848999731861332,\,\,
      \hat a_2 = -0.13818331109451658084,\,\,
 \hat a_3 = 0. 
 \end{equation}
 These are floating point approximations to
 algebraic numbers of degree $15$ over $\mathbb{Q}$.
 An exact representation is given by their minimal
 polynomials. For the first coordinate, this is
$$  \begin{small}
  \begin{matrix}
   62031304 a_1^{15}+2201861910 a_1^{14}
   +30829909776 a_1^{13}+206135547000 a_1^{12}+ 
      528436383696 a_1^{11} \\
      -1126661553720 a_1^{10}
      -9740892273264 a_1^9-4305524252579 a_1^8+ 
      26533957305582 a_1^7 \\ +88281552626154 a_1^6 
      +44254830057030 a_1^5-76332701171853 a_1^4 - 
      83490498412056 a_1^3 \\ +1857597611688 a_1^2
      +29825005557312 a_1+9354112703280 \qquad = \,\, 0.
      \end{matrix}
 \end{small}
 $$
 With this, the second coordinate $\hat a_2$ is a 
 certain rational expression in
$\mathbb{Q}(\hat a_1)$.
By plugging (\ref{eq:svalues}) and (\ref{eq:avalues}) 
into (\ref{QS_t_a}) with $t = 2/3$, we recover the
estimated probability table in (\ref{eq:pvalues}).
\hfill $\diamondsuit$
\end{example}

For larger cases, solutions to the likelihood equations (\ref{eq:scoreeqn})
are computed by iterative numerical methods, such as
 the {\it unidimensional Newton's method}.
 The updating equations at the $q$-th step of this iterative method are
 \begin{equation}
 \label{eq:iteration}
 a_i^{(q)} \,=\, 
 a_i^{(q-1)} -\frac{\partial \ell_{\bf n}({\bf a})/\partial a_i}{\partial^2 \ell_{\bf n}({\bf a})/\partial a_i^2}\bigl| {_{{\bf a}={\bf a}^{(q-1)}}}
 \quad \hbox{for} \quad
  i=1, \ldots, I-1 , \ q=1,2,\ldots \ .
 \end{equation}
 We find it convenient to rewrite the 
first derivatives (\ref{eq:scoreeqn}) as
\begin{equation}\label{lik_a}
\frac{\partial \ell_{\bf n}({\bf a})}{\partial a_i} = (1+t)\sum_{j=1}^I { \frac{s_{ij}}{2+(1-t)(a_i+a_j)}\left(1-\frac{1-t}{1+t}c_{ij}\right)\left(\frac{n_{ij}}{p_{ij}}-\frac{n_{ji}}{p_{ji}}\right) } .
\end{equation}
The second derivative equals
\begin{eqnarray}\label{Hessian_d}
\frac{\partial^2 \ell_{\bf n}({\bf a})}{\partial a_i^2} &=& -(1+t)\sum_{j=1}^I { \frac{2(1-t)s_{ij}}{[2+(1-t)(a_i+a_j)]^2}\left(1-\frac{1-t}{1+t}c_{ij}\right)\left(\frac{n_{ij}}{p_{ij}}-\frac{n_{ji}}{p_{ji}}\right) }\\
  && -(1+t)\sum_{j\neq i} { \frac{(1+t)s_{ij}^2}{[2+(1-t)(a_i+a_j)]^2}\left(1-\frac{1-t}{1+t}c_{ij}\right)^2\left(\frac{n_{ij}}{p_{ij}^2}+\frac{n_{ji}}{p_{ji}^2}\right) } .\nonumber
 \end{eqnarray}
 Here $i=1, \ldots, I-1$, the $p_{ij}$ are the expressions in (\ref{QS_t_a}), and 
\[c_{ij}=\frac{(1+t)(a_i-a_j)}{2+(1-t)(a_i+a_j)} \ .\]

We believe that the numerical
solution found by this iteration is always
the global maximum in (\ref{eq:unconstrainedMLE}).
This would be implied by the following conjecture, which holds for $t=0$ and $t=1$.

\begin{conjecture}
The Hessian $\,{\bf H}({\bf a}) = \left(\frac{\partial^2 \ell_{\bf n}({\bf a})}{\partial a_i \partial a_j}\right)$ is negative definite for all ${\bf a} \in \R^I$ with \eqref{eq:a_i parameters}. 
\end{conjecture}

We verified this conjecture for many examples with $t \in (0,1)$.
In each case, we also ran our iterative algorithm for many starting values,
and it always converged to the same solution.

The diagonal entries of the Hessian matrix are given in (\ref{Hessian_d}), while the non-diagonal are
\begin{eqnarray}\label{lHessian_nd}
\frac{\partial^2 \ell_{\bf n}({\bf a})}{\partial a_i \partial a_j}  &=&  \frac{2(1-t)^2 s_{ij} c_{ij}}{[2+(1-t)(a_i+a_j)]^2} 
 \left(\frac{n_{ij}}{p_{ij}}-\frac{n_{ji}}{p_{ji}}\right)  \\
 && + \ \frac{(1+t)^2 s_{ij}^2}{[2+(1-t)(a_i+a_j)]^2}  \left[1-\left(\frac{1-t}{1+t} c_{ij}\right)^2\right]
\left(\frac{n_{ij}}{p_{ij}^2}+\frac{n_{ji}}{p_{ji}^2}\right) . \nonumber
\end{eqnarray}

In the iterative algorithm described above,
we had fixed the last parameter $a_I$ at zero.
This ensures identifiability, and it is done for simplicity.
The constraint $a_I = 0$ defines a reference point
for the other parameters $a_1,\ldots,a_{I-1}$.
Under this constraint, (\ref{marg_differ_a}) leads to
\[a_i = \frac{1}{1+t}\left(\frac{p_{i+}-p_{+i}}{x_{i+}} - \frac{p_{I+}-p_{+I}}{x_{I+}}\right) \quad \hbox{for} \quad  i=1, \ldots, I-1 .\]
This means that the contribution of category $i$ to marginal inhomogeneity is compared to the last category's contribution. Hence, in view of (\ref{marg_differ_a}), a
reasonable alternative constraint could be  $\sum_{j=1}^I \frac{x_{ij}}{x_{i+}}a_j = 0$.
This constraint calibrates
  each category's contribution to marginal inhomogeneity relative to the weighted average of all 
  $I$ categories.

\begin{remark}\rm
The iterative procedure described above for fitting the ${\rm QS_t}$
models was implemented by us
 in {\ttfamily R}. The algorithm works regardless of whether we impose the
  restriction $a_I=0$ or not.
We noticed that when imposing this constraint, the algorithm requires more iterations to converge. The convergence is also affected by the initial values ${\bf a}^{(0)}$ we used. A classical choice would be $a_i=0$ for all $i$, 
as this corresponds to complete symmetry. However, we observed that for ${\bf a}^{(0)}$
with coordinates 
$\frac{n_{i+}-n_{+i}}{n_{i+}+n_{+i}}$, $i=1,\ldots, I$, the convergence is faster.
\end{remark}

\begin{remark}\rm 
Here we consider the model parameter $t$ as fixed. Alternatively, it could be estimated from the data, as for the power-divergence logistic regression model in \cite{KA}.
\end{remark}

\medskip

\section{Quasisymmetric Independence}

A natural submodel of (\ref{S})
is the symmetric independence model (SI), which is given by
\begin{equation}\label{SI}
p_{ij}=s_i s_j \ , \quad i,j=1, \ldots, I .
\end{equation}
The $I$ parameters $s_i$ are non-negative and sum to $1$.
The corresponding probability tables 
${\bf p} = (p_{ij})$ are symmetric and have rank $1$.
The models of quasisymmetric independence  (${\rm QSI}_t$) can be defined analogously to the ${\rm QS}_t$ models, by measuring departure from (\ref{SI}). Namely,
replacing the symmetric probabilities $s_{ij}$ 
in (\ref{QS_t_a}) by the factored form in (\ref{SI}),  we get 
\begin{equation}\label{QSI_t}
p_{ij}=s_i s_j\left(1+\frac{(1+t)(a_i-a_j)}{2+(1-t)(a_i+a_j)}\right) \  , \ \ \  i\neq j, \ \ \  i,j=1, \ldots, I .
\end{equation}

The MLEs of the parameters of the SI model in (\ref{SI}) are
\begin{equation}\label{MLE_SI}
\hat s_i =\frac{n_{i+}+n_{+i}}{2n} \quad \hbox{for} \quad i=1, \ldots, I.
\end{equation}
These are also the MLEs of the $s_i$ parameters in the ${\rm QSI}_t$ model. The likelihood equations for ${\bf a}$ are as before, 
but with $p_{ij}$'s in (\ref{lik_a}) as defined in (\ref{SI}) and (\ref{QSI_t}). Their numerical solution 
can be computed with the  iterative procedure described in Section 4, adjusted accordingly.

\begin{remark}{\rm
In Proposition~\ref{prop1}, if we replace the models S and ${\rm QS}_t$ by SI and ${\rm QSI}_t$, then an analogous statement holds.
Thus,  we have $\text{SI}=\text {MH} \cap {\rm QSI}_t$ for each $t\in [0, 1]$. 
}
\end{remark}

Following the discussion in Section 3, it would be interesting to derive
the implicit equations for the model ${\rm QSI}_t$. At present, we have
a complete solution only for the special case $t=1$.  The quasisymmetric
independence model
${\rm QSI}_1$ is defined by the parametrization
\begin{equation}
\label{eq:qsi1}
p_{ij} \, = \, s_i s_j \cdot (1 + a_i - a_j) , \qquad  1 \leq i,j \leq I. 
\end{equation}
Alternatively,  $\{i,j\}$ could range over the edges of a graph $G$, 
as in Section 3. In the following result, whose proof we omit,
we restrict ourselves to the case of the complete graph $K_I$.

\begin{proposition}
The prime ideal of the ${\rm QSI}_1$ model in (\ref{eq:qsi1})
is generated by the following 
homogeneous quadratic polynomials (for any  choices of indices $i,j,k,\ell$ among $1,\ldots,I$): 
\begin{itemize}
\vspace{-0.1cm}
\item $(p_{ij}+p_{ji})^2-4p_{ii}p_{jj}$, \vspace{-0.13cm}
\item $p_{kk}(p_{ij}-p_{ji})+p_{ki}p_{jk}-p_{ik}p_{kj}$, \vspace{-0.13cm}
\item $(p_{ij}-p_{ji})(p_{jk}-p_{kj})+4(p_{jj}p_{ki}-p_{ji}p_{kj})$, \vspace{-0.13cm}
\item $p_{\ell i}(p_{jk}-p_{kj})+p_{\ell j}(p_{ki}-p_{ik})+ p_{\ell k}(p_{ij}-p_{ji})$, \vspace{-0.13cm}
\item $p_{i\ell}(p_{jk}-p_{kj})+p_{j\ell}(p_{ki}-p_{ik})+ p_{k\ell}(p_{ij}-p_{ji})$.
\end{itemize}
\end{proposition}

\smallskip

The general case where $t <  1$ differs from the $t=1$ case in that the prime
ideal of $ {\rm QSI}_1$  is no longer generated by quadrics.
Even for $I=3$, a minimal generator of degree $3$ is needed:

\begin{example} \label{ex:runningQSI} \rm
Fix $I= 3$. 
For general $t \in \R$, we consider the model
 (\ref{QSI_t}) with $p_{ii} = s_i s_i$ for $i=1,2,3$.
 Its ideal
is minimally generated by $7$ polynomials: $6$ quadrics and one cubic.
\hfill $\diamondsuit$
\end{example}

\section{Fitting the Models to Data}

We next illustrate the new models and their features on some characteristic data sets. 
The goodness-of-fit of a model is tested asymptotically by the likelihood ratio statistic.
The associated degrees of freedom for ${\rm QS}_t$ and ${\rm QSI}_t$ are $\,df({\rm QS}_t)=(I-1)(I-2)/2\,$ 
and $\,df({\rm QSI}_t)=(I-1)^2$, respectively.
As we shell see, the models in each family 
can perform either quite similar or differ significantly,
depending on the specific data under consideration.

A case of similar behavior is the classical vision example of Table \ref{vision_data}. The model of QS ($t=0$) has been applied on this data often in the literature, while  \citet{KP} applied Pearsonian QS.
 Both models provide a quite similar fit, namely ($G^2=7.27076$, $p$-value $=0.06375$) for ${\rm QS}_0$
 and ($G^2=7.26199$, $p$-value $=0.06400$) for  ${\rm QS}_1$.
 Here, $df=3$.

 The behavior of the ${\rm QS}_t$ models for $t \in (0, 1)$ is similar.
 The log-likelihood values  vary from $-16388.11444$ ($t=0$) to $-16388.11006$ ($t=1$) while the saturated log-likelihood is $-16384.47906$ (see Figure \ref{examples_plot}, left). 
 Table \ref{vision_data_MLE} gives 
 the MLEs of the expected cell frequencies under the models $QS_0$, $QS_1$ and $QS_{2/3}$.
 For $t=2/3$ we get $G^2=7.26234$, with $p$-value $=0.06399$.

\begin{figure}[h]
\centerline{
\begin{tabular}{cc}
\includegraphics[scale=0.35]{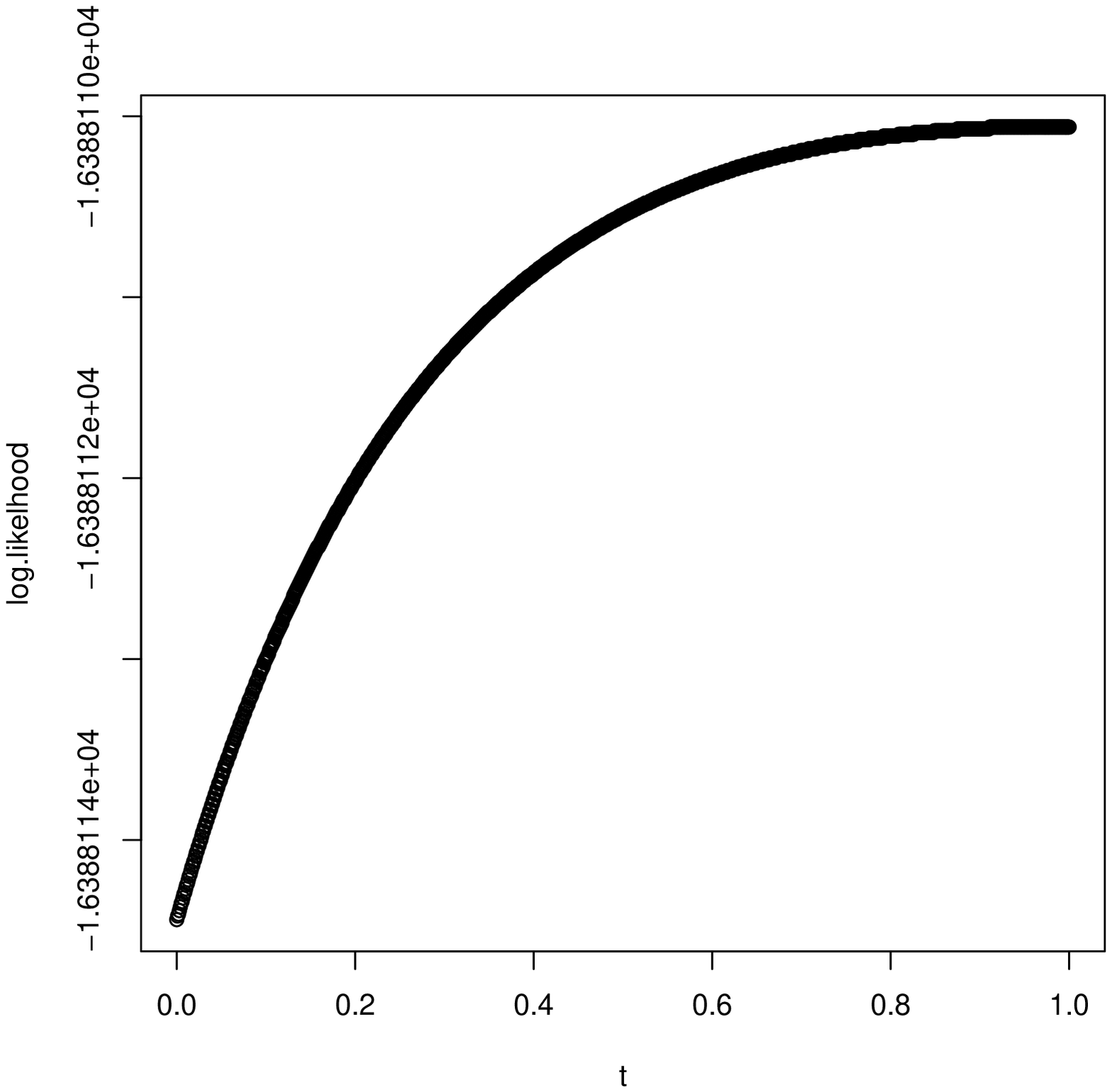}& \includegraphics[scale=0.35]{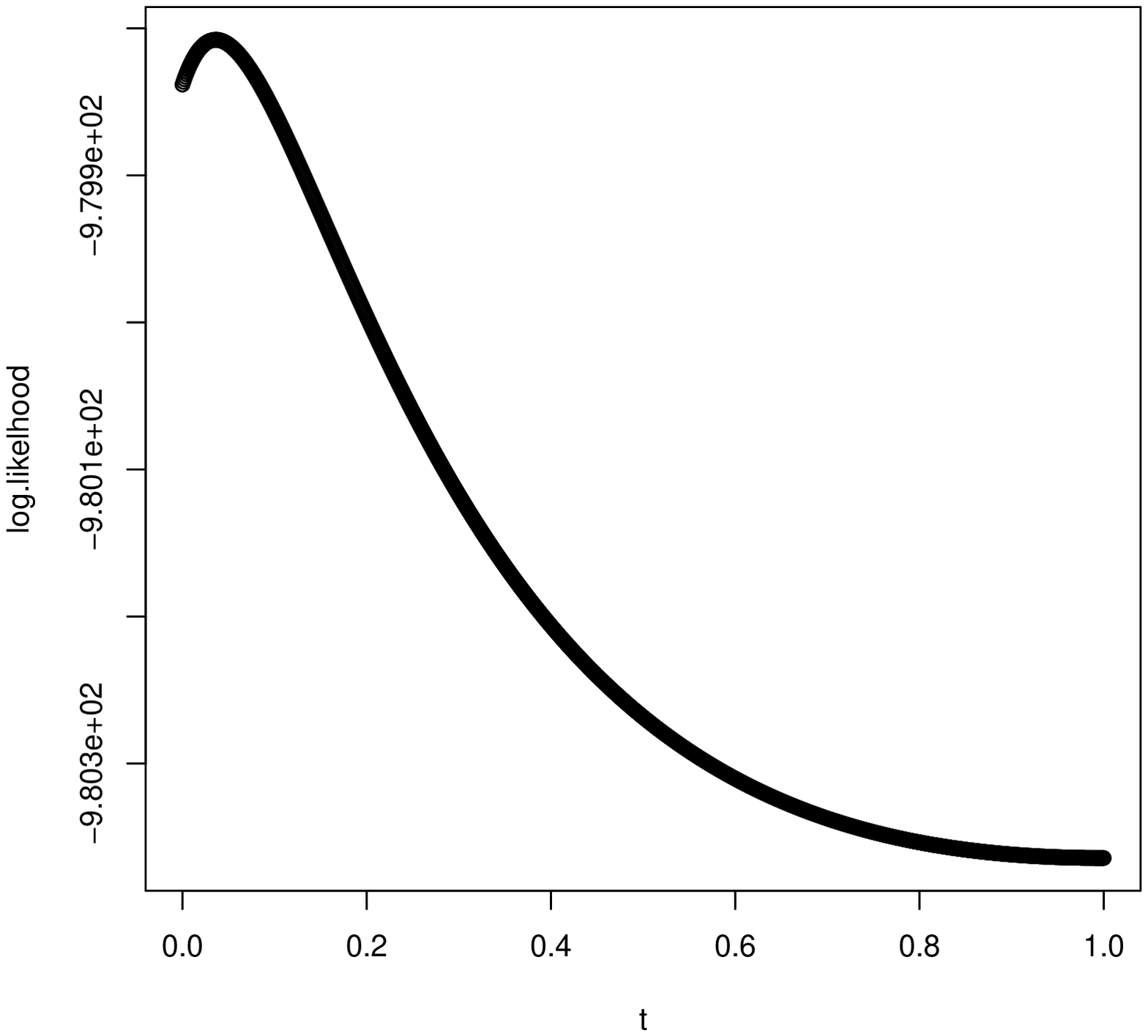}
\end{tabular}
}\vspace{-0.23in}
\caption{Log-likelihood values of ${\rm QS}_t$ for $t$ in $[0,1]$ for data in Tables \ref{vision_data} (left) and  \ref{examples_3x3_data}(c) (right).
} \vspace{-5pt}\label{examples_plot}\end{figure}

\begin{table}[h]

\begin{scriptsize}
\begin{center}
{\bf Left Eye Grade}  \\ 
\begin{tabular}{c@{~~}c@{~~}c@{~~}c@{~~}c}
\hline 
 {\bf Right Eye} & &  & & \\
 {\bf Grade} & best & second & third & worst \\

\hline
best & 1520  & 266 & 124 & 66  \\[-0.5ex]
			 &--&(263.38$^a$/~263.38$^b$/~263.39$^c$)&(133.58/~133.59/~133.60)&(59.04/~59.09/~59.09)
\\
second & 234 & 1512 &  432 &  78 \\[-0.5ex]
       &(236.62/~236.62/~236.61)&--&(418.99/~418.90/~418.90)&(88.39/~88.40/~88.40)
\\
third & 117 & 362 & 1772 & 205  \\[-0.5ex]
			 &(107.42/~107.40/~107.40)&(375.01/~375.10/~375.10)&--&(201.57/~201.58/~201.58)
\\
worst & 36 & 82 & 179 & 492  \\[-0.5ex]
       &(42.96/~42.91/~42.91)&(71.61/~71.60/~71.60)&(182.43/~182.42/~182.42)&--
\\ \hline
\end{tabular}
\end{center}
\end{scriptsize}
\vspace{-0.18in}
\caption{Unaided distance 
vision of right and left eyes for 7477 women.
Parenthesized values are ML estimates  of the expected frequencies under models ($a$) $QS_0$, (b) $QS_{2/3}$, and  ($c$) $QS_1$. 
\label{vision_data_MLE}} 
\end{table}

Examples for which the members of the ${\rm QS}_t$ family are not of similar performance are the two $3\times 3$ tables of \citet[Tables 3 and 4]{KP}, 
displayed in Table \ref{examples_3x3_data} (a) and (b).
\begin{table}[h]
\begin{center}
\begin{tabular}{ccccc} 
\begin{tabular}{c@{~~~}c@{~~~}c@{~~~}c}
{\bf (a)} &&& \\
\hline 
& {\bf 1} & {\bf 2} & {\bf 3} \\
\hline 
{\bf 1} & 28 & 10 & 15 \\
{\bf 1} & 122 & 126 & 102 \\
{\bf 1} & 49 & 22 & 26 \\
\hline
\end{tabular}
& \hspace{1cm} & 
\begin{tabular}{c@{~~~}c@{~~~}c@{~~~}c}
{\bf (b)} &&& \\
\hline 
& {\bf 1} & {\bf 2} & {\bf 3} \\
\hline 
{\bf 1} & 38 & 128 & 36 \\
{\bf 1} & 5 & 119 & 43 \\
{\bf 1} & 12 & 88 & 31 \\
\hline
\end{tabular}
& \hspace{1cm} & 
\begin{tabular}{c@{~~~}c@{~~~}c@{~~~}c}
{\bf (c)} &&& \\
\hline 
& {\bf 1} & {\bf 2} & {\bf 3} \\
\hline 
{\bf 1} & 28 & 12 & 25 \\
{\bf 1} & 122 & 126 & 102 \\
{\bf 1} & 49 & 22 & 26 \\
\hline
\end{tabular}
\end{tabular}
\end{center}
\vspace{-0.2in}
\caption{Simulated $3\times 3$ examples of \citet{KP}, generated by  the models (a)  ${\rm QS}_0$ and (b) 
${\rm QS}_1$ (their Tables 3 and 4, respectively).
A toy example in (c).
\label{examples_3x3_data}}
\end{table}
Here, the models ${\rm QS}_0$ and ${\rm QS}_1$ differ considerably
 in their fit. In particular, the data in Table \ref{examples_3x3_data} (a) are 
modeled well by ${\rm QS}_0$ but not by ${\rm QS}_1$ ($G^2_0=0.18572$ and $G^2_1=5.29006$), while the opposite holds for Table \ref{examples_3x3_data} (b), since $G^2_0=6.29035$ and $G^2_1=0.29215$. 

In such situations, the question arises whether some $t$ is appropriate for both data sets.
Finding $t$ such that ${\rm QS}_t$ works for  two or more $I\times I$ tables of the same set-up is of special interest in the study of stratified tables. 
Using  the same model on all strata makes parameter estimates among models comparable. This is a major advantage of the proposed family. 

Models that lie `in-between' the two extreme cases ($t=0$ and $t=1$) may lead to a consensus.
Even if that consensus model does not
perform as well as ${\rm QS}_0$ and ${\rm QS}_1$ on each table separately, it can provide a reasonable fit for both tables. 
To visualize this, Figure \ref{examples_3x3} (left) shows 
the $p$-values of the fit of the ${\rm QS}_t$ models with $t\in[0, 1]$, 
 for Tables \ref{examples_3x3_data} (a) and (b), by solid and dashed curves, respectively, along with the significance level of $\alpha=0.05$. The consensus model ${\rm QS}_t$ 
 would have
  $t\in(0.061, 0.302)$. Among these models, we propose  ${\rm QS}_{0.14}$, since the intersection of the two curves happens around $t=0.137$. The fit of this model for Table \ref{examples_3x3_data} (a) is $G^2=2.27614$ ($p$-value=0.1314) while for (b) 
 it is $G^2=2.16744$ ($p$-value=0.1409). The vector of MLEs for parameters $a_i$ is $(-0.5458, 1.8555, 0)$ and $(2.1247, -0.5406, 0)$, respectively. 
 We note that,  in deriving the consensus model,
   the $G^2$ values could have been used 
 as an alternative to the $p$-values  in Figure \ref{examples_3x3}.

\begin{figure}
\centerline{
\begin{tabular}{cc}
\includegraphics[scale=0.43]{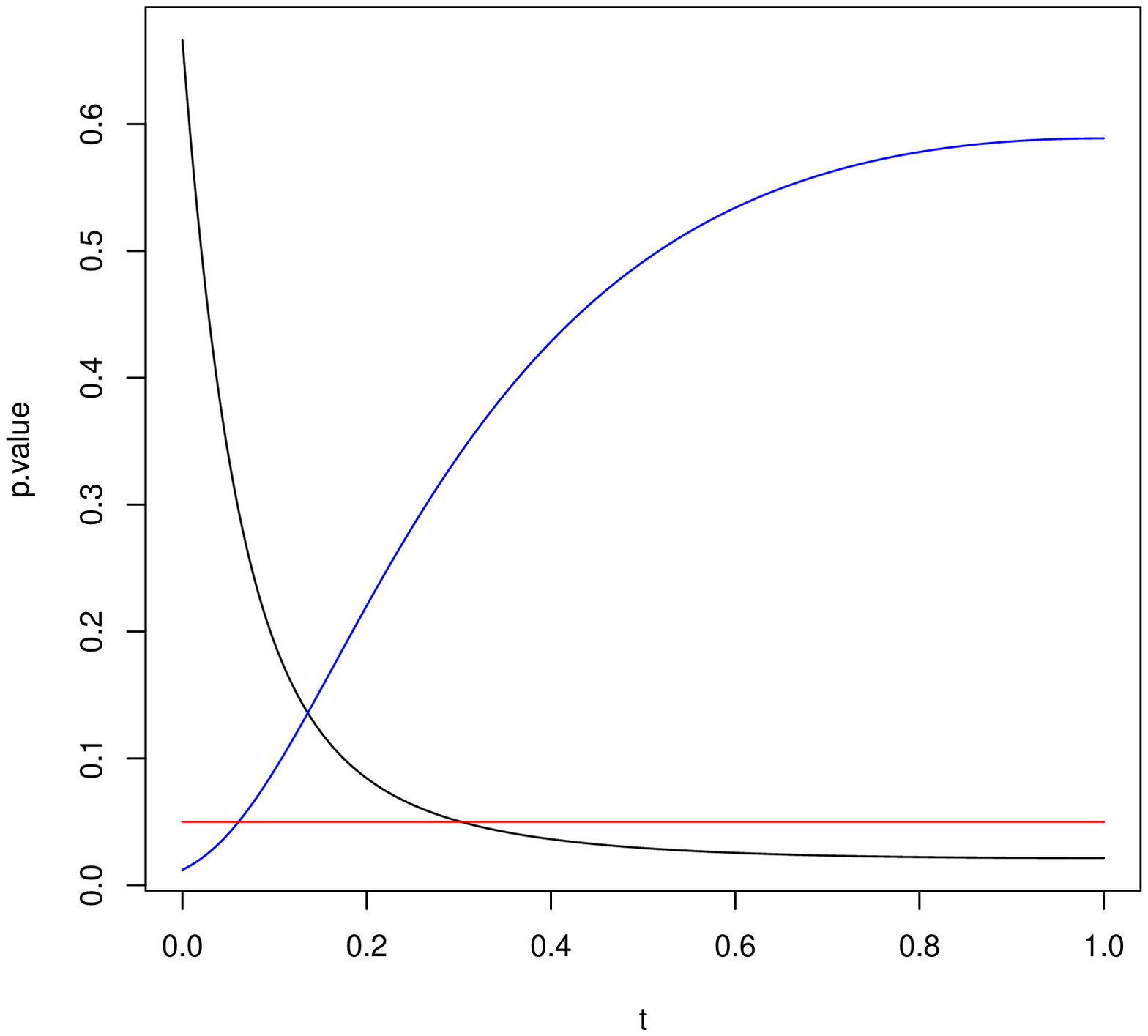}& \includegraphics[scale=0.43]{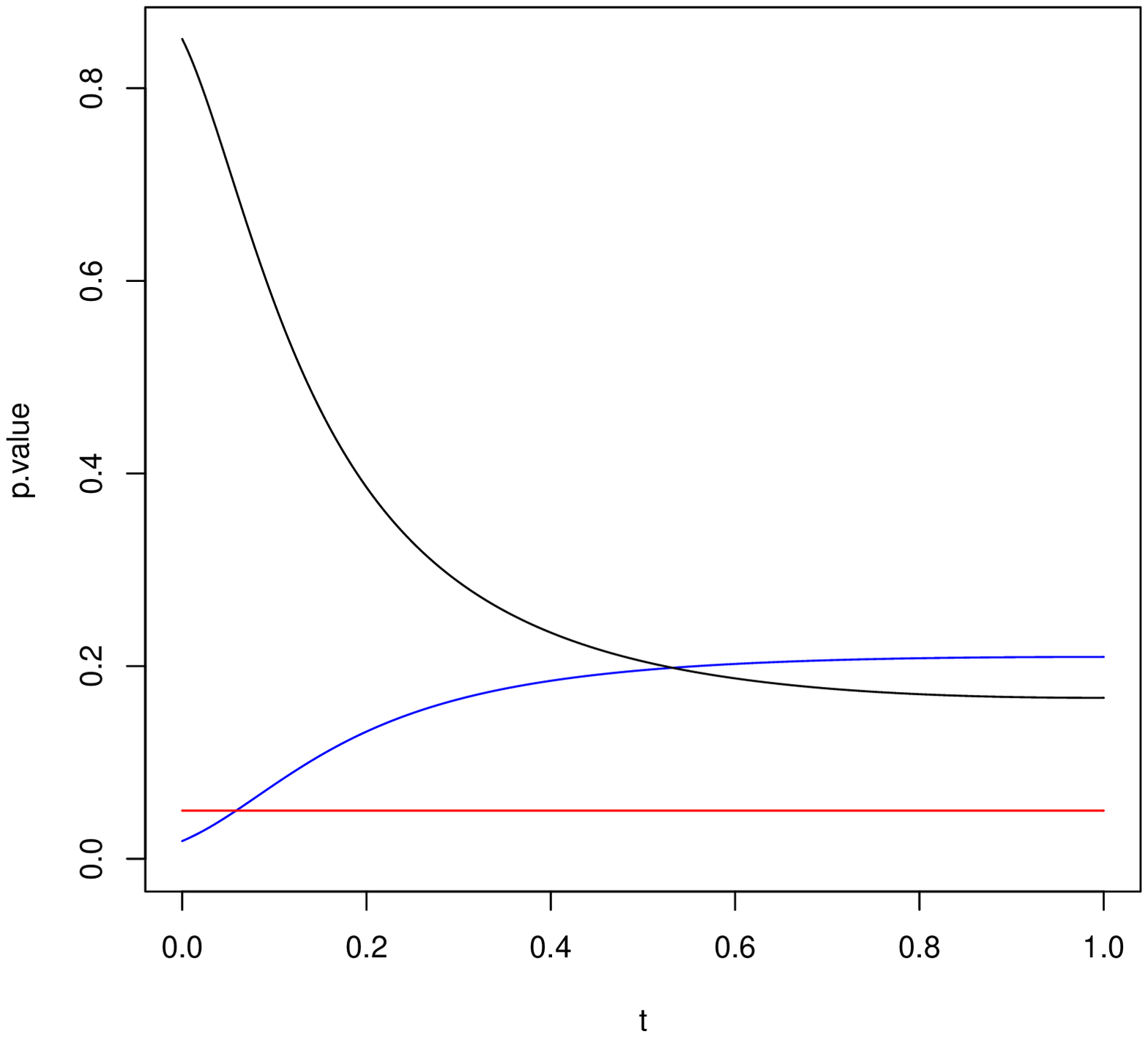}
\end{tabular}
}
\vspace{-0.3in}
\caption{
$p$-values for the $G^2$ goodness-of-fit test of ${\rm QS}_t$ (left) and ${\rm QSI}_t$ (right) for $t \in [0,1]$, along with the significance level $\alpha=0.05$.
Data are from Table \ref{examples_3x3_data}: (a) solid and (b) dashed.}
\label{examples_3x3}\end{figure}

\begin{figure}[h]
\begin{center}
\begin{tabular}{cc}
\includegraphics[scale=0.37]{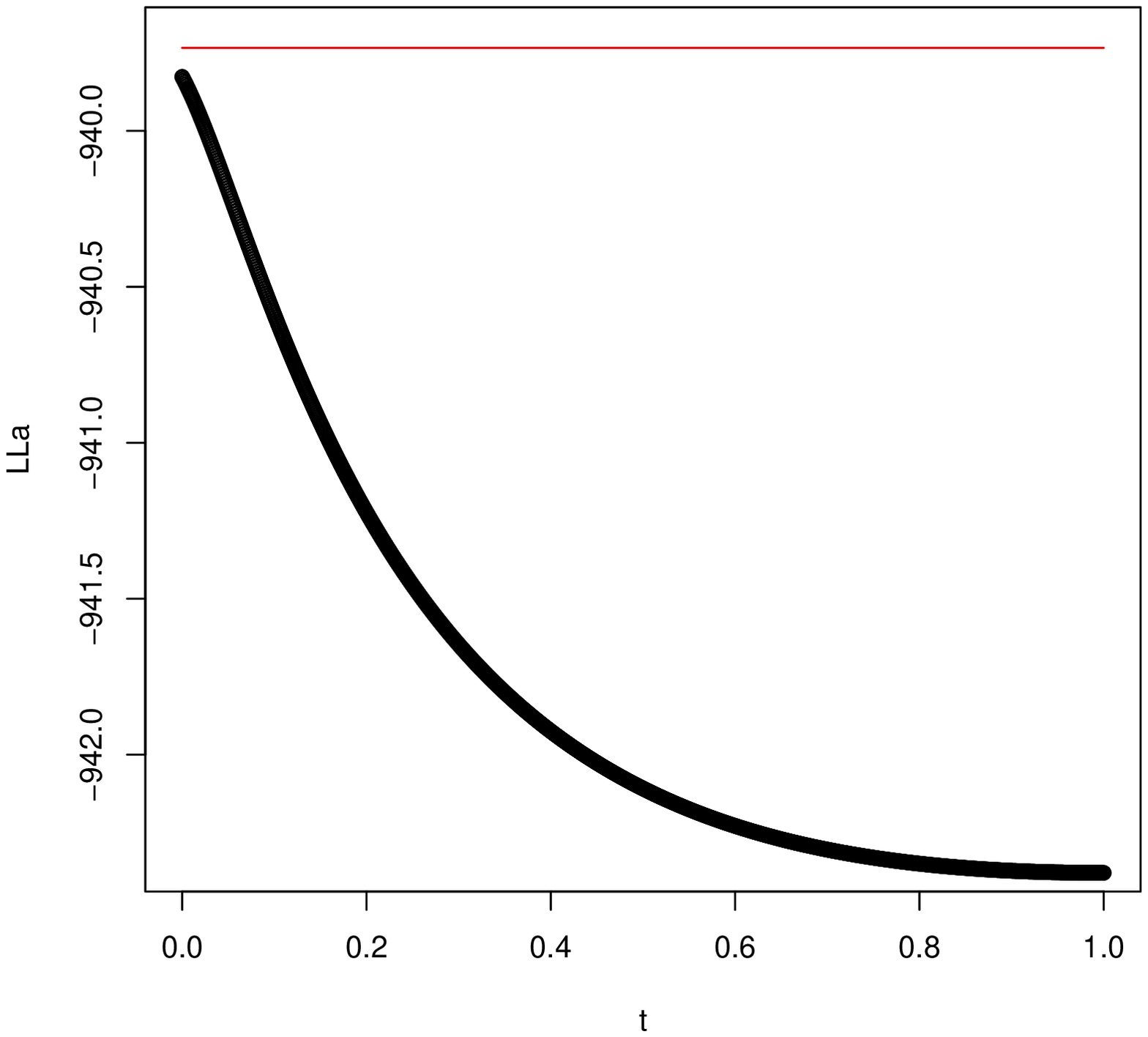}& \includegraphics[scale=0.37]{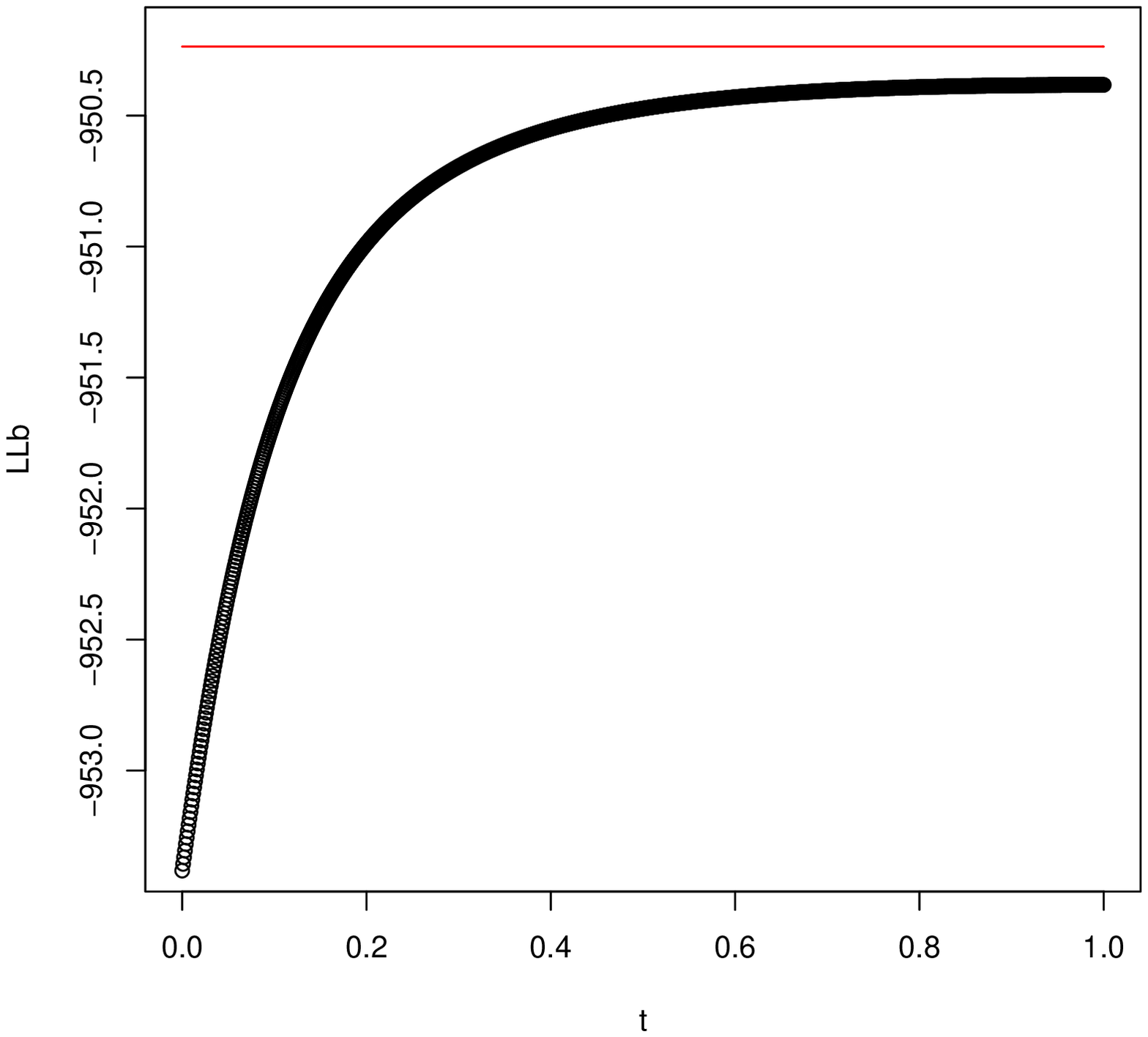} 
\vspace{-0.34in} \\
\includegraphics[scale=0.37]{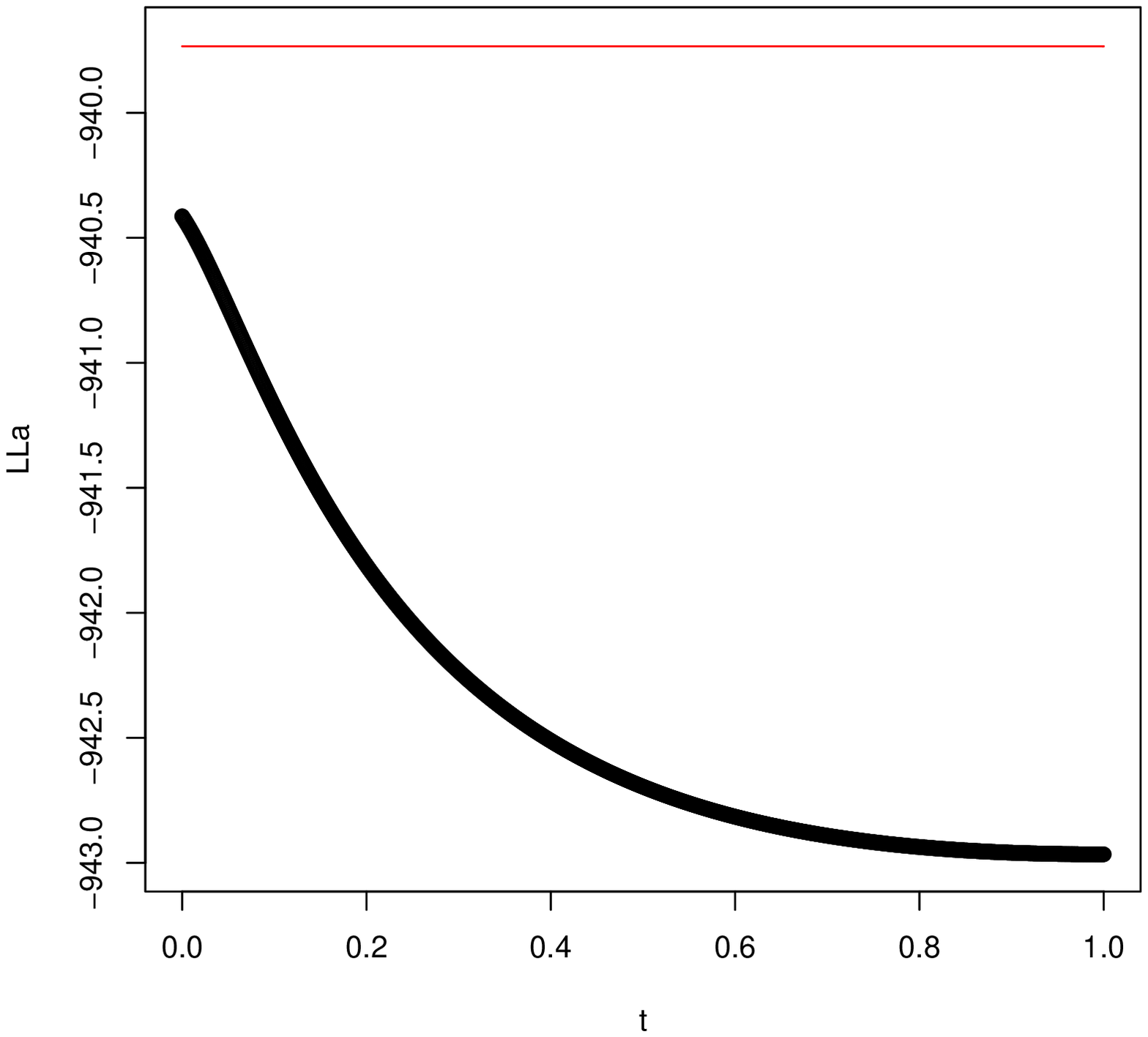}& \includegraphics[scale=0.37]{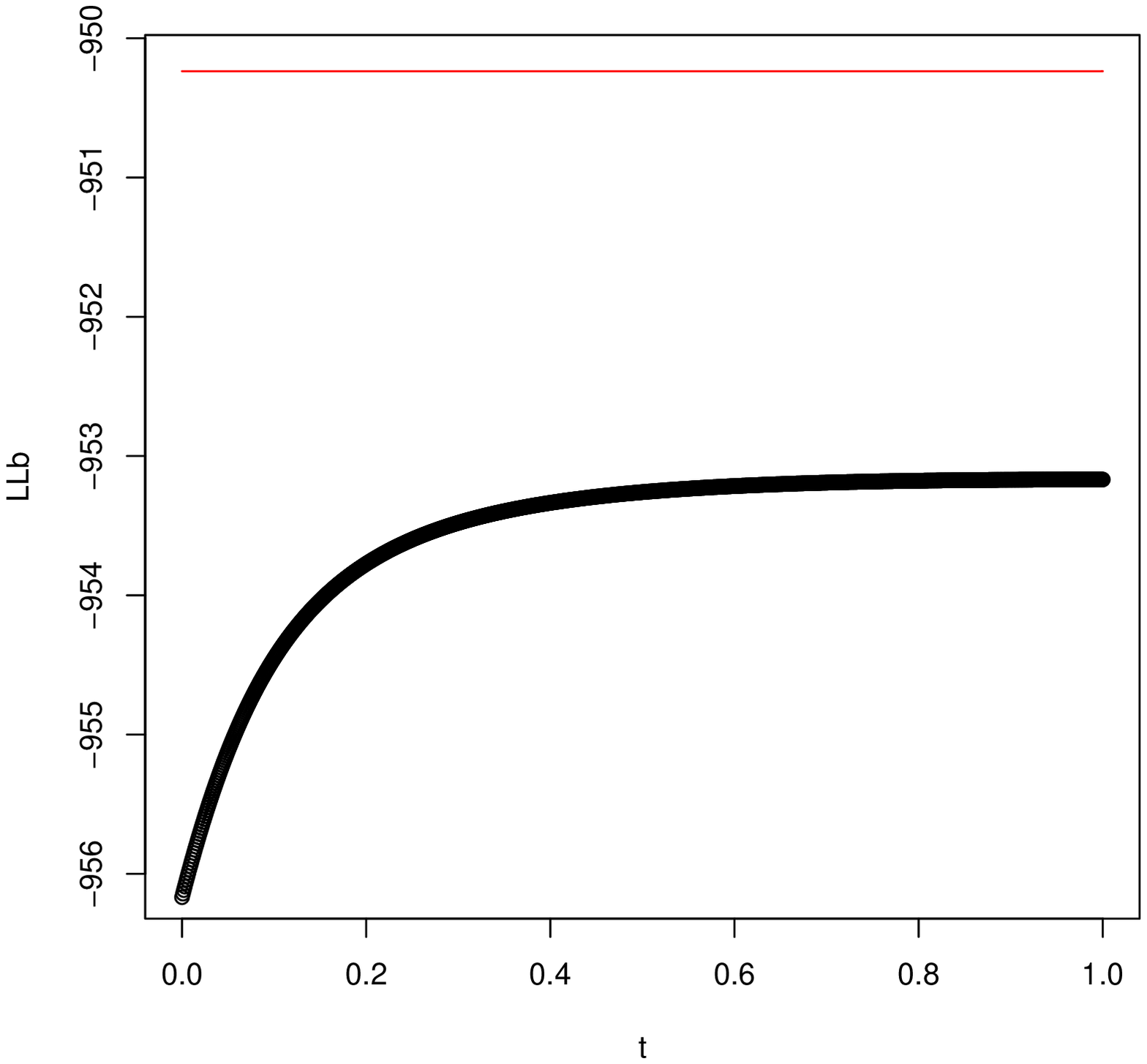} 
\end{tabular}
\end{center}
\vspace{-0.34in}
\caption{
Log-likelihood values of ${\rm QS}_t$ (upper) and ${\rm QSI}_t$ (lower) with $t \in [0,1]$ for
the data in Table \ref{examples_3x3_data}
 (a, left) and (b, right).  The straight line marks the saturated log-likelihood value.
\label{examples33LL_plot}} \vspace{-1ex}
\end{figure}

In all examples treated so far, the log-likelihood under ${\rm QS}_t$ was monotone in $t$ (see Figure \ref{examples_plot}, left, and Figure \ref{examples33LL_plot}, upper), 
suggesting that the `best' model will be achieved at either  $t=0$ or  $t=1$. This is not always the case. For example, for the data in Table \ref{examples_3x3_data} (c), the best fit occurs for $t=0.036$ (see also Figure \ref{examples_plot}, right), giving $G^2=1.742943\cdot 10^{-6}$ ($p$-value=0.9989) while for $t=0$ and $t=1$, it is $G^2=0.0610$ ($p$-value= 0.8049) and $G^2=1.1131$ ($p$-value=0.2914), respectively. 
Furthermore, even when the best model is for $t=0$ or $t=1$, we may still want to use  some $t\in (0, 1)$,
{\it e.g.}~for stratified tables with different optimal model at each level of the stratifying variable, as explained~above.

Applying the quasisymmetric independence models to Tables \ref{examples_3x3_data} (a) and (b), we observe that  $QSI_0$ fits well on Table \ref{examples_3x3_data} (a) but not on (b), while model $QS_1$ is of acceptable fit for both data sets. Indeed, we have $G_a^2(QSI_0)=1.3600$ ($p$-value=0.8511), $G_b^2(QSI_0)=11.8622$ ($p$-value=0.0184), $G_a^2(QSI_1)=6.4643$ ($p$-value=0.1671) and $G_b^2(QSI_1)=5.8640$ ($p$-value=0.2095). For the performance of the ${\rm QSI}_t$ model for $t \in [0, 1]$, see Figure \ref{examples_3x3} (right) and Figure \ref{examples33LL_plot} (lower). For $t=0.532$, the $p$-value of the fit of the model is equal to 0.1983 for both data sets.

All the examples of this section were worked out with {\tt R} functions we developed for fitting the ${\rm QS}_t$ and ${\rm QSI}_t$ models via the unidimensional Newton's method. 
The adopted inferential approach is asymptotic. In cases of small sample size, exact inference can be carried out via algebraic 
computations along the lines described in Section 3, and demonstrated in Examples 
\ref{ex:primenumbers} and \ref{ex:primenumbers2}.

\section{Divergence Measures}

The one-parameter family of QS models we proposed, ${\rm QS}_t$, $t \in [0, 1]$, connects the classical QS
model ($t=0$) and the Pearsonian QS model ($t=1$). These two belong both to a broader class of generalized QS models
that are derived using the concept of {\em $\phi$--divergence} \citep{KP, Pardo}. 
Measures of divergence quantify the distance between two probability distributions and play an important role in information theory and statistical inference. A well known divergence measure is the Kullback-Leibler (KL) divergence. However there exist broader classes of divergences. Such a class, including the KL as a special case, is the $\phi$-divergence. In the framework of two-dimensional contingency tables, 
this class is defined as follows. 

Let ${\mathbf p}=(p_{ij})$ and ${\mathbf
q}=(q_{ij})$ be two discrete bivariate probability
distributions. The {\em $\phi$--divergence} between ${\mathbf p}$
and ${\mathbf q}$ (or Csiszar's measure of information in
${\mathbf q}$ about ${\mathbf p}$) is defined~by
\begin{equation}
\label{phi-divergence} 
D_{\phi}({\mathbf p}, {\mathbf q})\,\,=\,\,\sum_{i,j}{q_{ij}\phi(p_{ij}/q_{ij})} .
\end{equation}
Here $\phi : [0, \infty) \rightarrow \mathbb{R}^+$ is a convex function such that $\,\phi(1)=\phi'(1)=0$, 
$\,0 \cdot \phi(0/0)=0$, and $\,0 \cdot \phi(x/0)= x \cdot \lim_{u\rightarrow\infty} \phi(u)/u$. For $\phi(u)=u\log (u)-u+1$ and $\phi(u)=(u-1)^2/2$, 
the divergence (\ref{phi-divergence}) becomes the KL and the Pearson's divergence, respectively. We adopt the notation in \citet{Pardo}. For properties of $\phi$-divergence, as well as a list of well-known divergences belonging to this family, we refer to \citep[Section 1.2]{Pardo}.
The differential geometric structure of the
Riemannian metric induced by such a divergence function is studied by \citet{AC}.

The generalized QS models introduced by \citet{KP} are based on the $\phi$-divergence and are characterized by the fact that each model in this class is the closest model to symmetry S, when the distance is measured by the corresponding divergence measure.  The classical QS model corresponds to the KL divergence, while the Pearsonian QS
corresponds to Pearson's distance.
  We shall prove in Theorem \ref{phi_div} that the other members of the ${\rm QS}_t$ family, {\it i.e.}~for $t \in (0,1)$, are  $\phi$-divergence QS models as well, and we identify the corresponding $\phi$ function. 

\begin{theorem}\label{phi_div}
Fix  $t \in (0, 1)$ and consider
the class of models that preserve the given row (or column) marginals $p_{i+}$ (or $p_{+i}$)
for $i=1,\ldots, I$, and also preserve the given sums $p_{ij}+p_{ji}=2s_{ij}$ for
$i, j=1,\ldots, I$.
In this class,
the ${\rm QS}_t$ model \eqref{QS_t_a}
is the closest model to the complete symmetry model S in \eqref{S},
where `closest' refers  the $\phi$-divergence  defined by
\begin{equation}\label{phi_function}
\begin{matrix}
& \phi(u) & = & f_t(u)-f_t(1)-f_t'(1)(u-1) ,
\smallskip \\
\hbox{where} & f_t(u)& =&(u+\frac{2t}{1-t}) \log(u+\frac{2t}{1-t}).\quad
\end{matrix}
\end{equation}
\end{theorem}

\begin{proof} We set $F_t(u)=\phi'(u)=\log(u+\frac{2t}{1-t})-\ell_t$, where $\ell_t=\log(1+\frac{2t}{1-t})$ is just a constant for given $t$.
This choice of constant ensures $\phi'(1)=0$. Then the inverse function
to $F_t$ is
\[
F_t^{-1}(x)=(\frac{-2t}{1-t})+e^{x+\ell_t} . \vspace{-.4cm}
\]
\vspace{-.2cm}
With this, we can write
\begin{eqnarray*}
p_{ij}&=&s_{ij}F_t^{-1}(\alpha_i+\gamma_{ij})\,=\,s_{ij}(\frac{-2t}{1-t}+e^{\alpha_i+\gamma_{ij}+\ell_t})
\,=\,s_{ij}\big(\frac{-2t}{1-t}+ \frac{\beta_i(\frac{2(1+t)}{1-t})}{\beta_i+\beta_j}\big) \ ,
\end{eqnarray*}
where   \vspace{-.3cm}
\[
\beta_i\,=\,e^{\alpha_i+\ell_t}\quad{\rm and }\quad e^{\gamma_{ij}}\,=\, \frac{\frac{2(1+t)}{1-t}}{e^{\alpha_i+\ell_t}+e^{\alpha_j+\ell_t}}\ .
\]
We next rewrite $p_{ij}$ as \vspace{-.2cm}
\[p_{ij}\,\,=\,\,s_{ij}\big(1+\frac{-(1+t)}{1-t}+ \frac{\beta_i(\frac{2(1+t)}{1-t})}{\beta_i+\beta_j}\big)
=s_{ij}\big(1+ \frac{\frac{(1+t)}{1-t}(\beta_i-\beta_j)}{\beta_i+\beta_j}\big)\ .
\]
%
Setting $\beta_i=1+(1-t)a_i$ and $\beta_j=1+(1-t)a_j$, this  translates into our parametrization
(\ref{QS_t_a}).
Now the result follows from \citet[Theorem~1]{KP}. For a probability table ${\bf s}$ 
with symmetry S,  the quantity
$D_{\phi}({\bf p},{\bf s})$ is minimized when ${\bf p}$ is   the probability table satisfying
 ${\rm QS}_t$.
\end{proof}

The fact that the ${\rm QS}_t$ models are $\phi$-divergence QS models implies that they share all the 
desirable properties of the $\phi$-divergence QS models \citep{KP}.
This includes the properties that highlight the physical interpretation issues of these models.
As far as we know, the $\phi$-divergence for the parametric $\phi_t$ function (\ref{phi_function}) has not been considered so far. Its study can be the subject of further research. Such a future project has the potential to
build a bridge between information geometry \citep{AC}
and algebraic statistics \citep{DSS}.
\bigskip \medskip

{\footnotesize
\noindent
{\bf Acknowledgements}.
Fatemeh Mohammadi was supported by  the Alexander von Humboldt Foundation.  \\
Bernd Sturmfels was supported by the NSF (DMS-0968882) and DARPA (HR0011-12-1-0011). 
}

\bigskip \bigskip

{
\noindent {Authors' addresses:}

\smallskip

\noindent Maria Kateri, Institute of Statistics, RWTH Aachen University, 52056 Aachen, Germany,\\
{\tt kateri@stochastik.rwth-aachen.de}

\smallskip

\noindent Fatemeh Mohammadi, Institut f\"ur Mathematik, Universit\"at Osnabr\"uck, 49069 Osnabr\"uck, \\
Germany,  {\tt fatemeh.mohammadi716@gmail.com}

\smallskip

\noindent Bernd Sturmfels,  University of California, Berkeley, CA 94720, 
USA, {\tt bernd@berkeley.edu}
}
\end{document}